\documentclass[12pt,oneside,english,a4paper]{amsart}
\usepackage[T1]{fontenc}
\usepackage[utf8]{inputenc}
\usepackage[dvipsnames]{xcolor}
\usepackage[numbers,sort,compress]{natbib}
\usepackage{comment,todonotes}
\usepackage{bm}
\usepackage{amstext,amsthm,amssymb}
\usepackage{xparse,mathrsfs,mathtools}
\usepackage[english]{babel}
\usepackage{tikz}
\usepackage{esint}
\usepackage{float}
\usepackage{subcaption}
\usepackage{placeins}
\usepackage{lmodern, microtype}  
\usepackage{geometry}

\newtheorem*{thm*}{Theorem}
\newtheorem{thm}{Theorem}[section]

\newtheorem{lem}{Lemma}[section]
\makeatletter\let\c@lem\c@thm\makeatother

\newtheorem{cor}{Corollary}[section]
\makeatletter\let\c@cor\c@thm\makeatother

\newtheorem{prop}{Proposition}[section]
\makeatletter\let\c@prop\c@thm\makeatother

\makeatletter\let\c@Claim\c@thm\makeatother

\makeatletter\let\c@question\c@thm\makeatother

\newtheorem{Def}{Definition}[section]
\makeatletter\let\c@Def\c@thm\makeatother

\newtheorem{Conj}{Conjecture}[section]
\makeatletter\let\c@Conj\c@thm\makeatother

\theoremstyle{remark}
\newtheorem{rem}{Remark}[section]
\makeatletter\let\c@rem\c@thm\makeatother
\newtheorem*{rem*}{Remark}

\usepackage[pdftex,unicode=true,%
pdfusetitle,pdfdisplaydoctitle=true,%
pdfpagemode=UseOutlines,%
bookmarks=true,bookmarksnumbered=false,bookmarksopen=true,bookmarksopenlevel=2,%
breaklinks=true,%
pdfencoding=unicode,psdextra,
pdfcreator={},
pdfborder={0 0 0}
]{hyperref}

\usepackage[capitalise,noabbrev]{cleveref}

\crefname{thm}{Theorem}{Theorems}
\Crefname{thm}{Theorem}{Theorems}
\crefname{lem}{Lemma}{Lemmas}
\Crefname{lem}{Lemma}{Lemmas}
\crefname{cor}{Corollary}{Corollaries}
\Crefname{cor}{Corollary}{Corollaries}
\crefname{prop}{Proposition}{Propositions}
\Crefname{prop}{Proposition}{Propositions}
\crefname{Claim}{Claim}{Claims}
\Crefname{Claim}{Claim}{Claims}
\crefname{question}{Question}{Questions}
\Crefname{question}{Question}{Questions}
\crefname{Def}{Definition}{Definitions}
\Crefname{Def}{Definition}{Definitions}
\crefname{Conj}{Conjecture}{Conjectures}
\Crefname{Conj}{Conjecture}{Conjectures}
\crefname{rem}{Remark}{Remarks}
\Crefname{rem}{Remark}{Remarks}
\crefname{equation}{Equation}{Equations}
\Crefname{equation}{Equation}{Equations}

\newcommand{\NN}{\mathbb{N}}
\newcommand{\ZZ}{\mathbb{Z}}
\newcommand{\QQ}{\mathbb{Q}}
\newcommand{\RR}{\mathbb{R}}
\newcommand{\CC}{\mathbb{C}}
\newcommand{\PP}{\mathbb{P}}

\newcommand{\FFF}{K}
\newcommand{\FFFF}{K_{\bullet, D}}

\newcommand{\QQfield}{\mathbb{Q}(\sqrt{-D})}

\newcommand{\hfat}{\mathfrak{h}}

\newcommand{\Tgaus}{\mathcal{G}}
\newcommand{\Cgaus}{\mathcal{T}}
\newcommand{\Mpartition}{\mathscr{P}}
\newcommand{\myprojector}{\mathscr{P}}
\newcommand{\mynihl}{\mathscr{N}}
\newcommand{\myid}{\mathbb{I}}
\newcommand{\BanachHolo}{\mathscr{C}(\mathscr{P})}

\newcommand{\ringoel}{\mathscr{O}_L}
\newcommand{\ringint}{\mathcal{O}_K}

\newcommand{\Myoperator}{\mathscr{L}}
\newcommand{\Aopera}{\mathscr{K}}
\newcommand{\Myoperatordual}{\mathscr{L}^\star}
\newcommand{\ExpX}{\mathbb{E}_X}

\newcommand{\tick}[1]{\bm{#1}}
\newcommand{\alphares}{\langle \tick{\alpha} \rangle}
\newcommand{\fatalpha}{\bm{\alpha}}

\newcommand{\LandauO}{\mathcal{O}}

\newcommand{\vol}{{\text{vol}}}
\newcommand{\boar}[1]{\left[ #1 \right]}

\DeclareMathOperator{\sgn}{\text{sgn}}

\renewcommand\Re{\operatorname{Re}}
\renewcommand\Im{\operatorname{Im}}

\title{The Limiting distribution of elliptic Dedekind sums}
\author{M. BORDIGNON AND P. MINELLI}

\address{Dipartimento di Matematica “F. Enriques”, Università degli Studi di Milano, Via Saldini 50, 20133 Milano, Italy.}
\email{matteo.bordignon@unimi.it}

\address{Institute for Analysis and Number Theory, TU-Graz, Kopernikussgasse 24, 8010 Graz, Austria}
\email{minelli@tugraz.at}


\begin{document}

\begin{abstract}
    We consider elliptic Dedekind sums that were introduced by Sczech as generalizations of the classical ones to complex lattices. We prove that these sums -suitably normalized- have a Gaussian limiting distribution. As an application, we prove a conjecture due to Ito \cite{ItoDed}.
\end{abstract}

\maketitle

\section{Introduction}
Let $h,k$ be integers with $k> 0$. \textit{Dedekind sums} are defined as
\[
\mathfrak{s}(h,k) \;=\; \sum_{l=1}^{k-1}
  \left(\!\left(\frac{l}{k}\right)\!\right)
  \left(\!\left(\frac{hl}{k}\right)\!\right),
\]
where
\[
((x)) \;=\;
\begin{cases}
  x - \lfloor x \rfloor - \tfrac{1}{2} & \text{if } x \notin \mathbb{Z}, \\
  0 & \text{if } x \in \mathbb{Z},
\end{cases}
\]
is the usual sawtooth function. It is possible to show that $\mathfrak{s}$ depends only on the quotient of its two arguments, that is $\mathfrak{s}:\QQ\backslash\{0\} \to \RR$, by $\frac{h}{k}\to \mathfrak{s}(\frac{h}{k}).$ Therefore, from now on, we will tacitly assume $\gcd(h,k)=1$. These sums appeared for the first time in Dedekind's investigation of the transformation of the $\eta$ function under the group $\text{SL}_2(\mathbb{Z})$ \cite{DedekindErl}.
The Dedekind eta function is given by
\[
\eta(\tau) \;=\; e^{\pi i \tau/12}
  \prod_{n=1}^{\infty}\bigl(1 - e^{2\pi i n\tau}\bigr),
  \qquad \tau \in \mathbb{H}.
\]
Under the action of a linear fractional transformation given by an element
$A= \bigl(\begin{smallmatrix} a & b \\ c & d \end{smallmatrix}\bigr)
\in \mathrm{SL}_2(\mathbb{Z})$ with $c > 0$, the $\eta$ function transforms (see e.g. \cite[Chapter 3]{Apostol1990}) as
\[
\eta\!\left(A\tau\right)
  \;=\; \varepsilon(a,b,c,d)\cdot(-i(c\tau+d))^{1/2}\cdot\eta(\tau),
\]
where the multiplier is given by
\[
\varepsilon(a,b,c,d)
  \;=\; \exp\!\left(\pi i\left(\frac{a+d}{12c} - \mathfrak{s}(d,c)\right)\right),
\]
or, written in logarithmic terms,
\begin{align}\label{eq:logtransform}
\log \eta(Az)=\log \eta(z)+ i\pi\left(\frac{a+d}{12c}-\mathfrak{s}(d,c)\right)+\frac{1}{2}\log\left(-i(cz+d)\right).
\end{align}
From the above relations it is possible to extract the well-known \textit{reciprocity formula}
\begin{align}\label{eq:reclawded}
    \mathfrak{s}(h,k)+\mathfrak{s}(k,h)=\frac{1}{12}\left(\frac{h}{k}+\frac{k}{h}+\frac{1}{hk}\right)-\frac{1}{4},
\end{align}
valid for positive integers $h,k$ with $\gcd(h,k)=1$.  These sums have been extensively studied. We briefly mention the result of Hickerson \cite{Hickerson}, who showed that the set $\{\mathfrak{s}(\frac{h}{k}): \frac{h}{k}\in \mathbb{Q}\}$ is dense in the real line and the set $\{((\frac{h}{k}), \mathfrak{s}(\frac{h}{k})): \frac{h}{k}\in \QQ\}$ is dense in the plane. Concerning distributional results, Vardi \cite{Vardiequidistribution} showed that for any real $r>0$, the set
\[
\{ r\mathfrak{s}(h,k) \text{ mod } 1: 0<h<k\leq X \text{ with } \gcd(h,k)=1\}
\]
equidistributes in $[0,1)$ as $X\to\infty$. Moreover, he proved that, under a suitable normalization, Dedekind sums possess a limiting distribution that is \textit{Cauchy},  see \cite{Vardi}. More formally, let us denote $\Omega_Q:=\{\frac{h}{k}\in \mathbb{Q}\cap [0,1], \, 1\leq k\leq Q\}$, equipped with the uniform probability measure. Vardi proved the following.

\begin{thm}\label{thm:Vardi}
   For any $\epsilon>0$ and  $Q$ large enough, we have
\begin{align*}
\mathbb{P}_Q\left(z\in \Omega_Q: \frac{\mathfrak{s}(z)}{\log Q}\leq \frac{x}{2\pi}\right)= \frac{1}{\pi}\int_{-\infty}^x \frac{1}{1+y^2} \,dy +\LandauO\left(\frac{1}{(\log Q)^{\frac{1}{5}-\epsilon}}\right),
\end{align*} 
 uniformly in $x\in \mathbb{R}$.
\end{thm}
In both these results, Vardi exploited the link with the spectral theory of automorphic forms. We refer the reader to the thematic monograph \cite{RadeGross} and the more up-to-date survey \cite{Girstmairsurvey} for a comprehensive overview of various results regarding these sums.

\medskip
There are also several possible generalizations of Dedekind sums; see for instance \cite{Apostolgen, Zagier} or the work of Beck \cite{Beck2003}, which unifies many of the previous generalizations. A few decades ago, Sczech \cite{Sczechinv} introduced an \textit{elliptic} version of Dedekind sums, where $\mathbb{Z}$ is essentially replaced by the ring of integers $\ringint$ of an imaginary quadratic number field $K=\QQfield$, with $D>0$ a square-free integer. These sums, which are known in the literature as \textit{Sczech sums} or \textit{elliptic sums}, share many properties with their classical analogue. In order to better describe the construction performed by Sczech, let us recall that for coprime $h,k \in \mathbb{Z}$, $k \neq 0$, we have the following alternative representation of Dedekind sums:
\begin{align}\label{def:dedekindcotdef}
\mathfrak{s}(h,k)=\frac{1}{4k}\sum_{j=1}^{k-1}\cot\left(\pi\frac{hj}{k}\right)\cot\left(\pi\frac{j}{k}\right).
\end{align}
Let now $K$ be an imaginary quadratic field and let $L\subset \mathbb{C}$ be a non-degenerate lattice. Let $\ringoel=\{m\in \mathbb{C}: mL\subset L\}$ be the ring of multipliers. This is either the ring of integers $\ringint$ of the field $K$ or an order. From now on, we will consider only the cases where $K$ has class number one and $L$ is selected such that $\ringoel=\ringint$.
Define, for $k\in \mathbb{Z}_{\geq 0}$ and $z\in \mathbb{C}$, the functions 
\begin{align}\label{eq:Gkintro}
    G_k(z)=\sum_{\substack{w\in L\\ z+w\notin L}} (z+w)^k| w+z|^{-s}\Big|_{s=0},
\end{align}
the value above being given by analytic continuation. For $k=1$, the functions above are analogous to the $\cot(\pi z)$. Therefore, mimicking \cref{def:dedekindcotdef}, Sczech defined, for $a, c\in \ringint$, $c\neq 0$, the sum
\begin{align*}
\mathfrak{D}(a,c)=\frac{1}{c}\sum_{r\in L/cL} G_1\left(\frac{ar}{c}\right)G_1\left(\frac{r}{c}\right).
\end{align*}
He also proved that the sum above is nontrivial whenever $D\notin \{1, 3\}$, while $\mathfrak{D}(a,c)=0$ for all $a,c\in \ringint$ with $ac\not=0$ if $D=1$ (Gauss integers) or $D=3$ (Eisenstein integers). Moreover, for two coprime elements $a, c\in \ringint$, one has the following analogue of the reciprocity law in \eqref{eq:reclawded}
\begin{align*}
    \mathfrak{D}(a,c)+\mathfrak{D}(c,a)=2iG_2(0)\Im\left(\frac{a}{c}+\frac{c}{a}+\frac{1}{ac}\right),
\end{align*}
where $\Im(z)$ denotes the imaginary part of $z$ and $G_2(0)\neq 0$ provided $D\notin\{1, 3\}$. A further relevant analogy with the classical case is the existence of a function $H: \mathbb{H}^3\to \CC$ on the upper half space, transforming under $\text{SL}_2(\ringint)$, whose transformation law is described by a relation analogous to \eqref{eq:logtransform}. However, unlike the classical case, the construction of such a function is highly non trivial and postdates the work of Sczech, see \cite{Ito-logeta}.

It is possible to show, see \cite[p. 200]{ItoDed}, that the normalized quantity
\begin{align*}
\tilde{\mathfrak{D}}(a,c):=\frac{1}{i\sqrt{|D|} G_2(0)}\mathfrak{D}(a,c),
\end{align*}
is rational. Once again, $\tilde{\mathfrak{D}}$ can be seen as a function of the quotient of the two arguments: $\tilde{\mathfrak{D}}: K\backslash\{0\}\to K$ by $\frac{a}{c}\to \tilde{\mathfrak{D}}(\frac{a}{c})$. Ito \cite{ItoDed} proved an analogue of Hickerson's density theorem: when $K$ is Euclidean and the sum $\mathfrak{D}$ is nontrivial (i.e., $D\in\{2,7,11\}$), the set $\{(\tilde{\mathfrak{D}}(\frac{a}{c}), \frac{a}{c}): \frac{a}{c}\in K\}$ is dense in $\mathbb{R}\times\mathbb{C}$. This result was subsequently generalized to any imaginary quadratic field $K$ with $D\notin\{1,3\}$ by Bartel et al. \cite{Bartelletall}. Recently, an analogue of Vardi's equidistributional result has been established by Klinger, Logan and Wong \cite{Klinger-LoganWong}, using spectral methods.
For the remainder of this paper, we restrict our attention to the class number one case, moreover, whenever not differently specified, we will assume $D\in \{2, 7, 11\}$. Consider the Farey set
\begin{align}\label{def:fareyset}
   \FFF_{D}(X) := \FFF(X) := \left\{ \frac{a}{c} \in K : |c|^2 < X \right\}.
\end{align}
 Seeking to understand the distribution of $\tilde{\mathfrak{D}}$ over the set $\FFF_D(X)$, Ito \cite{ItoDed} performed numerical experiments (for $D=2$) and formulated the following conjecture.

\begin{Conj}[Ito]\label{eq:Itoconj}
Let $B_2:=\{z=x+iy: 0\leq x<1,\ 0<y\leq \frac{1}{\sqrt{2}}\}$. Then, for $D=2$,
\begin{align*}
    \lim_{X\to \infty} \frac{1}{\#(\FFF_2(X)\cap B_2)}\sum_{\frac{a}{c}\in \FFF_2(X)\cap B_2} \big\vert\tilde{\mathfrak{D}}\left(a,c\right)\big\vert=\infty.
\end{align*}
\end{Conj}
To put the above claim in perspective, recall that for classical Dedekind sums one can show (see \cite{Girstmair}) that there are constants $C_1, C_2>0$ with $C_2=2C_1$ such that
\begin{align*}
   (C_1+o(1))\log^2 X \leq \frac{1}{\#\{\frac{h}{k}\in \mathbb{Q}\cap [0,1]: k<X\}} \sum_{\substack{\frac{h}{k}\in \mathbb{Q} \cap [0,1]\\ k<X}} |\mathfrak{s}(h,k)| \leq (C_2+o(1))\log^2 X.
\end{align*}
To the best of our knowledge, no asymptotic formula is known. The best result until now, from which the double average result above descends, is due to Girstmair and Schoissengeier \cite{Girstmair}, who proved
\begin{align}\label{eq:girstmairdouble}
    \frac{1}{\pi^2}\log^2 k + o(\log^2 k)\leq \frac{4}{\varphi(k)}\sum_{\substack{1\leq h\leq k\\ \gcd(h,k)=1}} |\mathfrak{s}(h, k)| \leq  \frac{2}{\pi^2}\log^2 k + o(\log^2 k),
\end{align}
their key achievement being the lower bound, while the upper bound follows from Hickerson's formula together with pre-existing results tracing back to Heilbronn \cite{Heil}. Girstmair and Schoissengeier also conjectured that the correct asymptotic should match the upper bound, as discussed in the introduction of \cite{Girstmair}. Prior to Girstmair and Schoissengeier, asymptotics for even moments of Dedekind sums were obtained by Conrey et al. \cite{Conrey}, who proved 
\begin{align}
\sum_{\substack{1\leq h\leq k\\\gcd(h,k)=1}}\vert\mathfrak{s}(h,k)\vert^{2m}\sim  C_m(k) k^{2m},
\end{align}
where $C_m(k)$ is an Archimedean factor. This result has since seen several improvements concerning the error term; see for instance \cite{Wenpeng}. Very recently, a ``shifted'' version of Conrey's result was obtained by Borda, Munsch and Shparlinski \cite{BMSparigor}, while for a generalized version see \cite[Theorem 3]{BettinEsterman}.
\subsection{Statement of results}

For $D=2$, define
\begin{align}\label{eq:exagonaldomain1}
    I_{2}:=\Bigl\{ z=x+iy: |x|\leq \frac{1}{2},\ |y|\leq \frac{1}{\sqrt{2}}\Bigr\}.
\end{align}
For $D=7,11$, define
\begin{align}\label{eq:exagonaldomain2}
    I_{D}:=\Bigl\{ z=x+iy: |x|\leq \frac{1}{2},\ \bigl|y\pm x/\sqrt{D}\bigr|\leq \frac{D+1}{4\sqrt{D}}\Bigr\}.
\end{align}
Set
\begin{align*}
\FFFF(X):=\Bigl\{ z=\frac{a}{b}\in K: 1\leq |b|^2 < X\Bigr\}\cap I_D,
\end{align*}
and equip $\FFFF(X)$ with the uniform probability measure $\mathbb{P}_X$.
The goal of the present paper is to establish a limiting distribution 
for suitably normalized Sczech sums $\tilde{\mathfrak{D}}(a,c)$. 
The main result - namely \cref{thm:maintheorem} below - is an 
analogue of \cref{thm:Vardi}, with a remarkable difference: the 
limiting distribution is \textit{Gaussian} rather than \textit{Cauchy}, and is also highlighted in \cref{fig:histogram} below.
\begin{thm}[The limiting distribution of Sczech sums]
\label{thm:maintheorem}
Let $X$ be sufficiently large and let $D\in \{2,7, 11\}$. There is a constant $\kappa_D>0$ 
such that, uniformly in $x\in\mathbb{R}$,
\begin{align*}
    \mathbb{P}_X\!\left( z\in \FFFF(X): 
    \frac{\tilde{\mathfrak{D}}(z)}{\kappa_D\sqrt{\log X \log\log X}} 
    < x\right) 
    = \frac{1}{\sqrt{2\pi}} \int_{-\infty}^x e^{-y^2/2}\, dy 
    + \LandauO\left(\frac{1}{(\log\log X)^{1-\epsilon}}\right).
\end{align*}
\end{thm}

As an immediate corollary we obtain the following result.
\begin{cor}\label{cor:ourcorollary}
\cref{eq:Itoconj} holds true.
\end{cor}

\begin{rem}
Although Ito's conjecture follows as a corollary from our main theorem, it appears inaccessible via the methods of \cite{Conrey,Girstmair}. This contrasts sharply with the one-dimensional case, where elementary circle method arguments suffice to obtain lower bounds - and even asymptotics - of the correct order of magnitude. Indeed, in higher dimensions, the contribution from the Farey points - which carry the main contribution in the $\mathfrak{s}$ case - is dominated by the error term. As the reader is encouraged to verify, a similar phenomenon also occurs when attempting to adapt the strategy of \cite{Conrey} to obtain an analogue of \eqref{eq:girstmairdouble} for $|\mathfrak{s}(h,k)|^{1/2}$.
\end{rem}

\begin{rem}
    Evidence that the sums $\tilde{\mathfrak{D}}$ exhibit different behaviour than their classical counterpart was also conjecturally known to Ito. Indeed, he conjectured that the \textit{restricted} average of these sums, namely 
    \begin{align*}
        \frac{1}{\#\FFF_2(X)}\sum_{z \in \FFF_2(X)\cap B_2} \tilde{\mathfrak{D}}\left(z \right)
    \end{align*} 
    is bounded, while for the classical case one has
    \begin{align*}
        \lim_{Q\to \infty} \frac{1}{\#\Omega_Q} \sum_{x\in \Omega_Q\cap [0, \frac{1}{2}]} \mathfrak{s}(x)=\infty.
    \end{align*}
    Ito's claim for the classical case has been established in a strong quantitative sense by the second named author, Sourmelidis and Technau in \cite{MST1} and \cite{MSTII}.
\end{rem}

\begin{figure}[H]
    \centering
    \includegraphics[width=0.8\linewidth]{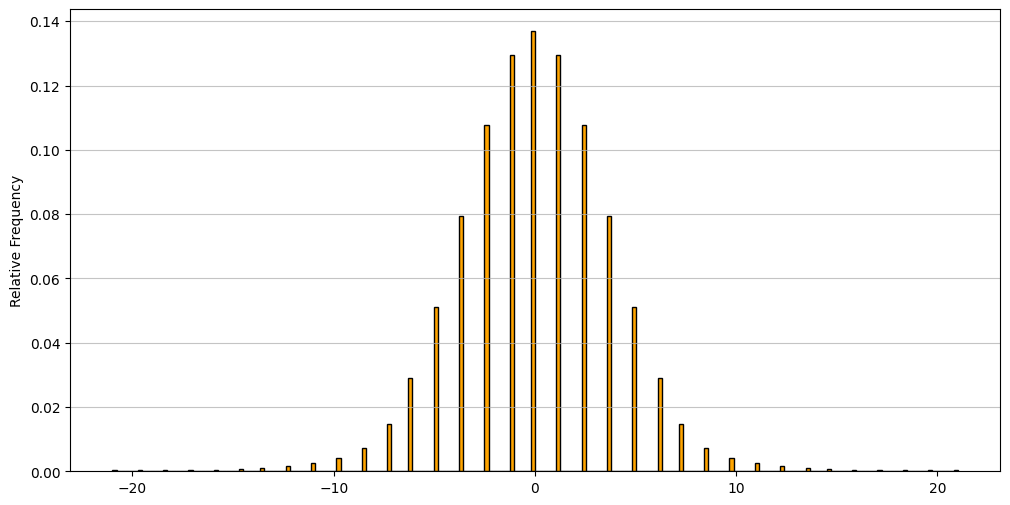}
    \caption{Plot of $\tilde{\mathfrak{D}}(a,c)/\sqrt{\log X\log\log X}+o(1)$ over the set $\FFFF(X)$ for $X=2500$ and $D=2$.}
    \label{fig:histogram}
\end{figure}

\subsection{Overview}
The proof of Vardi's result \cref{thm:Vardi} rests crucially on a link between Dedekind sums and Kloosterman sums. Thus, the spectral theory of automorphic forms comes into play in handling certain sums of Kloosterman sums (see also \cite{Sarnak-Goldfeld-sums}). On the other hand, relatively recently, \cref{thm:Vardi} was reproved, as the outcome of a remarkably general result, by Bettin and Drappeau \cite{BettinDrappeau-main}. Their approach --- which is dynamical in nature --- builds upon a well-established framework. We close this introductory section by presenting a minimal overview of this framework, which will also serve us in outlining the main steps in the proof of \cref{thm:maintheorem}. In doing so, we gambit generality for clarity.
\subsubsection{Quick overview of the literature}
Let $\Cgaus$ be the usual Gauss map and let $f:[0,1]\to \RR$ be a function of bounded variation that is not a coboundary, i.e is not of the form $g-g \circ \Cgaus+C$ for some function of bounded variation $g$ and some constant $C$. One can study the distribution of the quantity
\begin{align*}
    S_N(f, x):= \sum_{j=1}^N f(\Cgaus^j(x)),
\end{align*}
for $x$ chosen uniformly in $[0,1]$. It is possible to prove (see \cite{Broise}) that the quantity above has the following limiting distribution as $N\to\infty$: one has convergence towards a standard Gaussian distribution with an effective rate of convergence, in the following sense
\begin{align*}
    \mathbb{P}\left(\frac{S_N(f,x)-c(f)N}{\sigma_f\sqrt{N}}\leq y\right) = \Psi(y)+\mathcal{O}_f\left(\frac{1}{\sqrt{N}}\right),
\end{align*}
where
\begin{align*}
    \Psi(y):=\frac{1}{\sqrt{2\pi}}\int_{-\infty}^{y} e^{-t^2/2}\,dt,
\end{align*}
is the cumulative distribution function of a standard Gaussian. 
\par
\noindent
Turning now to the \textit{rational} case, namely the case where $x\in \Omega_N$ is sampled uniformly at random, one can consider the quantity
\begin{align}\label{eq:S-var-definition}
    S(f,x):=\sum_{j=1}^{r(x)} f(\Cgaus^{j}(x)),
\end{align}
where $r(x)$ is the smallest $r$ such that $T^{r+1}(x)=0$. In other words $x=[0; a_1, \dots, a_{r(x)}]$, where $a_j:=\lfloor 1/\Cgaus^{j-1}(x)\rfloor$ for $1\leq j\leq r(x)$. In this context, pioneering work has been carried out by Baladi and Vallée \cite{BV-Euclidean, BV-exponential} and Vallée \cite{Vall-euclidean-analysis}. Baladi and Vallée proved that if $f$ is \textit{of moderate growth} (cf.\ \cite{BV-Euclidean}), which essentially means $f(x)\ll \log(1/x)$, and \textit{constant by parts}, that is constant on each interval $[\frac{1}{n+1}, \frac{1}{n})$, for $n\in \NN_{>0}$, then a central limit theorem can be extracted. 

\begin{thm}[Baladi and Vallée, cf.\ {\cite[Theorem 4]{BV-Euclidean}}]\label{thm:Baladi-Vallee-steps}
The quantity $S(f,x)$, suitably normalized, has Gaussian limiting distribution, that is, for any $y\in \RR$ one has
\begin{align*}
    \mathbb{P}_Q\left(x\in \Omega_Q: \frac{S(f,x)-c(f)\log Q}{\sigma(f)\sqrt{\log Q}}\leq y\right) = \Phi(y)+ \mathcal{O}_f\left(\frac{1}{\sqrt{\log Q}}\right),
\end{align*}
where $c(f)$ and $\sigma(f)>0$ are explicit real constants.
\end{thm}
It is worth mentioning that Hensley \cite{Hensley} obtained the above theorem for the special case $f=1$, with a suboptimal error term $O(1/\log^{1/24} Q)$. This amounts to showing that the number of steps in the classical Euclidean algorithm, after normalization, converges in distribution to a Gaussian. \cref{thm:Baladi-Vallee-steps} above improves the error term to the optimal rate and extends the result to a wider class of division algorithms and functions $f$.
\par
\noindent
Bettin and Drappeau \cite{BettinDrappeau-main} extended the above result to a wider class of functions. First, they allowed a higher level of complexity in \eqref{eq:S-var-definition}: let $(f_i)_{1\leq i\leq k}$ be piecewise $C^1$ functions $[0,1]\to \RR$. Define the \textit{period law}
\begin{align*}
    \tick{f}(x):=\sum_{j=1}^k f_j(\Cgaus^{j-1}(x))
\end{align*}
and the quantity
\begin{align*}
    S(\tick{f}, x):=\sum_{j=1}^{r(x)} f_j(\Cgaus^{j-1}(x)),
\end{align*}
where the indices of $f_j$ are read modulo $k$. They obtained distributional results for $S(\tick{f}, x)$ for functions $(f_j)_{1\leq j\leq k}$ which are not necessarily constant by parts, and also not necessarily log-bounded. In particular, they proved that if two technical conditions are satisfied \cite[see conditions (3.7) and (3.8)]{BettinDrappeau-main}, then one can extract an asymptotic expansion for the characteristic function of $S(\tick{f}, \cdot)$ near $0$. The limiting distribution is then characterized in terms of $\tick{f}$, see \cite[Theorem 3.1]{BettinDrappeau-main}. This allowed them to consider quantities such as the sum of $\lambda$-th powers ($\lambda>0$) of the partial quotients
\begin{align*}
    x\mapsto \sum_{j=1}^{r(x)} a_j^\lambda,
\end{align*}
which corresponds to the selection $k=1$ and $f_1(x)=\lfloor 1/x \rfloor^\lambda$, as well as Dedekind sums, which correspond to the selection $f_1(x)=\lfloor 1/x \rfloor$, $f_2(x)=-f_1(x)$, $k=2$. In particular, they recovered \cref{thm:Vardi} with the improved exponent $1-\epsilon$ in the error term.

\begin{rem}
The results of Bettin--Drappeau \cite{BettinDrappeau-main} are more general 
than stated above: the authors work with  functions $f$ in H\"older spaces, allowing them to treat, for instance, the distribution of modular symbols associated with level~$1$, weight $k\geq 12$ cusp forms, recovering special cases of a theorem of Nordentoft \cite{Nordentoft} as well as a result of Constantinescu \cite[Theorem 1.4]{Constantinescu}. Their result was later employed by Drappeau and Nordentoft \cite{DrappeauNordentoft} for studying the distribution of central values of additive twists of Maaß forms; see also \cite[Section 3]{ConstantinescuNordentoftGafa} for a more detailed exposition of the results in \cite{BettinDrappeau-main} in relation with additive twists of $L$-series. In a similar direction, we mention the result of Lee and Sun \cite{LeeSun}, who reproved the main result of Petridis and Risager \cite{PetridisRisager} concerning the distribution modular symbols. Their approach, unlike the one originally employed in \cite{PetridisRisager}, is dynamical. Moreover, it was claimed in \cite[Remark 1.2]{Kimetall}, that dynamical methods like those used in \cite{Kimetall, LeeSun} can be used to study the distribution of Bianchi modular symbols in the hyperbolic 3-space, recovering the results of Constantinescu and Nordentoft \cite{ConstantinescuNordentoft}.
\end{rem}
\noindent
Recently, Kim, Lee and Lim considered the continued fraction expansion of a complex number $z$ with $\ZZ$ replaced by the ring of integers of a Euclidean imaginary quadratic number field, i.e.\ $\QQ(\sqrt{-D})$ for $D\in\{1,2,3,7,11\}$. Given $z\in \QQ(\sqrt{-D})\cap I_D$, we have $z=[0; \alpha_1, \dots, \alpha_{\ell(z)}]$ for some integer $\ell(z)\geq 1$, where we restrict to $D\in\{2,7,11\}$ for contextual reasons. Let $g:\mathcal{O}_D\to \RR_{\geq 0}$ be a function satisfying $g(\alpha)\ll \log |\alpha|$, which is precisely the \textit{moderate growth} condition of Baladi and Vallée. Define
\[
    S(g, z):=\sum_{j=1}^{\ell(z)} g(\alpha_j),
\]
for $z\in \mathbb{Q}(\sqrt{-D})\cap I_D$. They proved the following theorem.

\begin{thm}[Number of steps in the complex Euclidean algorithm, cf.\ {\cite[Theorem B]{Kimetall}}]\label{thm:Kimetall-analogue-Baladi-Vallée}
Let $\PP_X$ denote the uniform probability measure on $\FFFF(X)$. There are explicitly computable positive real constants $c(g)$ and $\sigma(g)$ such that for any fixed $x\in \RR$
\begin{align*}
    \PP_X\!\left( z\in \FFFF(X): 
    \frac{S(g, z)-c(g)\log X}{ \sigma(g) \sqrt{\log X}} 
    < x\right) 
    =\Phi(x)
    + \LandauO\left(\frac{1}{\sqrt{\log X}}\right).
\end{align*}
\end{thm}
\subsubsection{Overview of the techniques}
All three results cited above, namely \cref{thm:Baladi-Vallee-steps}, \cref{thm:Kimetall-analogue-Baladi-Vallée}, and \cite[Theorem 3.1]{BettinDrappeau-main}, exploit the connection between the characteristic function of $x\mapsto S(f,x)$ -the expectation being taken with respect to $\PP_X$ or $\PP_Q$ -and a Dirichlet series, which in turn can be realized as a special instance of a two-parameter family of operators acting on a selected indicator function. To give a rough sketch, let $Y$ be some Banach space of functions over $[0,1]$ (as in \cite{BV-Euclidean, BettinDrappeau-main}) or over $I_D$ (as in \cite{Kimetall} and in our setting). Define the operator $Y\to Y$ by
\begin{align*}
    \Myoperator f(z):= \sum_{z_0:\, \Tgaus(z_0)=z} \frac{1}{J_{\Tgaus}(z_0)} f(z_0),
\end{align*}
where $\Tgaus$ is defined as in \eqref{def:Tgausmap} and $J_\Tgaus$ denotes its Jacobian. One can then define a two-parameter family of operators $\Myoperator_{s,w}$, where $(s,w)$ lies in some open subset of $\CC\times \CC$ or $\CC\times \RR$ and with $\Myoperator_{1,0}=\Myoperator$. Here and in the sequel, unless stated otherwise, we assume $w$ to be real and denote it by $t$. This family of operators allows one to link the characteristic function of $S(g,z)$, namely $\chi_S(t)=\ExpX[e^{itS(z)}]$, with a Dirichlet series
\begin{align*}
    H(s,t)=\sum_{v\in K} \frac{1}{\hfat(v)^{2s}} e^{it S(v)},
\end{align*}
where for $v=\frac{a}{b}\neq 0$, $v\in \FFFF$, we set $\hfat(v)=|b|$.
Roughly, one can reduce the understanding of $\chi_S(t)$ to that of the integral
\begin{align*}
    \int_{\sigma_0-i\infty}^{\sigma_0+i\infty} X^{2s} H(2s,t)\, ds,
\end{align*}
where $\sigma_0$ is chosen so that $H(2s,t)$ converges absolutely, via a standard contour integration. At this point, the rigid structure of $H(2s,t)$ comes into play. It can be shown that
\begin{align}\label{eq:basic-link-intro}
    H(2s,t)= \left(\mathbb{I}-\Myoperator_{s,t}\right)^{-1} \mathfrak{A}_{s,t}[\mathbf{1}_{I_D}](0),
\end{align}
where $\mathfrak{A}_{s,t}$ is another operator $Y\to Y$ which is, for the moment, of secondary importance. It is easy to see that $H(2s,t)$ is analytic in $s$ whenever $\Re(s)>1$. However, what is required here is a \textit{zero-free strip} $\Re(s)>1-\delta$ for some $\delta>0$, for all $|t|\leq t_0$ and some $t_0>0$. For a general Dirichlet series, such a statement is typically completely out of reach unconditionally. However, the rigid structure \eqref{eq:basic-link-intro} makes it possible to exploit the spectral properties of the operator as well as its contracting properties with respect to a family of norms. At this stage, the most challenging step is to ensure good decay for $H(2(\sigma+i\tau),t)$ for large $|\tau|$.  This can be obtained by exploiting ``decay of correlation'' as in the work of Baladi and Vallee \cite{BV-Euclidean}, later adapted to the system $(I_D, \Tgaus)$ by \cite{Kimetall}. After this, one is left with an asymptotic expression for $\chi_S(t)$, whose behavior is encoded by the spectral properties of the right-hand side of \eqref{eq:basic-link-intro}. For reasonable functions $S$, the final asymptotics for $\chi_S(t)$ can be extracted using a various combinations of tools from perturbation theory, the asymptotic analysis of integrals and integral transforms.
\subsubsection{Our direction}
The present paper sits at the intersection of \cite{BettinDrappeau-main} and \cite{Kimetall}. The work of Kim, Lee and Lim will serve us in a fundamental way when defining the transfer operator. As we will explain in the upcoming sections, the system $(I_D, \Tgaus)$ is more intricate than $([0,1), \Cgaus)$ as $\Cgaus$ is not \textit{full branch}. On the other hand, as in \cite{BettinDrappeau-main}, we will work with period functions of the shape $\Im \boar{\frac{1}{z}}$, which are not of moderate growth, as it grows as $\vert \alpha\vert$ instead of $\log\vert \alpha\vert$. This is one of the reasons Hwang's quasi-power theorem (see e.g.\ \cite[Theorem~7.1]{Kimetall}) cannot be applied here. The condition we will need, which will appear \textit{en passant}, is an analogue of \cite[condition~(3.7)]{BettinDrappeau-main}. Luckily for us, the functions we work with are constant on parts, and their influence on $\Aopera_{s,t}$ does not require us to enlarge the delicate function space of \cite{Kimetall}. Therefore, a condition analogue to \cite[ condition (3.8)]{BettinDrappeau-main} holds almost automatically. The first part of the paper, where we establish the meromorphic continuation of $H(2s,t)$, follows with moderate effort, as we are able to borrow several ingredients from \cite{Kimetall}. The second part, where, as in \cite{BettinDrappeau-main}, the asymptotic analysis of a certain oscillatory integral is required, is more subtle. There, the arithmetic of $\ringint$ and the intricate nature of the invariant probability measure $d\mu$ associated with the system $(I_D, \Tgaus)$ require extra work compared to \cite{BettinDrappeau-main}. The computationally elusive nature of $d\mu$, whose density is given by an integral expression over a fractal domain, is the reason we do not provide an explicit value for $\kappa_D$ in \cref{thm:maintheorem}.

\subsection{Structure of the paper}
This paper is structured as follows. In \cref{sec:continued-fraction}, we present the relevant background on the complex continued fraction expansion and its connection to Sczech sums. In the preparatory \cref{sec:prep-material}, we introduce some technical constructions, mostly borrowed from \cite{Kimetall, Nakada}, and describe the Banach space of functions $\BanachHolo$ on which the family of operators $\Aopera_{s,t}$ acts. In \cref{sec:transoperator}, we define the transfer operator, which is a twist of the operator introduced in \cite{Kimetall}; in the same section, we establish the connection between the characteristic function $\chi_S(t)$ (see \eqref{eq:charfunctions}) and the function $H(2s, t)$. The spectral analysis of the operator, together with the analytic continuation of $H(2s,t)$, is carried out in \cref{sect:spectral-theory}. Asymptotics for $\lambda(s,t)$ and the quantity $s_0(t)-1$ are worked out in \cref{sec:lambda-asypt}, \cref{sec:expansion-char-function} and \cref{sec:integral-asypmt}. Finally, we prove \cref{thm:maintheorem} in \cref{sec:proof-of-the-theorem}.

\section*{Acknowledgments}
Both authors thanks Sandro Bettin, Noy Soffer-Aranov and Christopher Aistleitner for useful comments on an earlier draft.
\par
Some parts of the present work were carried out while M.B. was at the Department of Mathematics, Università di Milano,
as the recipient of a postdoctoral fellowship funded by the ``MUR - Department of Excellence 2023-2027, CUP code G43C22004580005 - project code DECC23\_012\_DIP''.
\par
Some parts of the present work were carried out while P.M. was partially supported by the Austrian Science Fund (FWF), grant-doi 10.55776/P35322. This research was funded in whole or in part by the Austrian Science Fund (FWF) 10.55776/PAT4579123. For open access purposes, the authors have applied a CC BY public copyright license to any author accepted manuscript version arising from this submission. 
\section{Continued Fraction Expansion}\label{sec:continued-fraction}

We begin by defining the fundamental regions for the lattice of algebraic integers. We recall that the ring of integers of the fields $\mathbb{Q}(\sqrt{-D})$ for $D \in \{2, 7, 11\}$ is
\[
\mathcal{O}_{\mathbb{Q}(\sqrt{-D})} = 
\begin{cases}
\mathbb{Z}[\sqrt{-2}] & D = 2, \\[6pt]
\mathbb{Z}\left[\frac{1+\sqrt{-7}}{2}\right] & D = 7, \\[6pt]
\mathbb{Z}\left[\frac{1+\sqrt{-11}}{2}\right] & D = 11.
\end{cases}
\]
Now we consider the three sets $I_D$ mentioned in the introduction: the rectangular domain
\[
I_{2} := \bigl\{ z = x + iy : |x| \leq 1/2,\; |y| \leq 1/\sqrt{2} \bigr\}
\]
and the two hexagonal domains
\[
I_{D} := \Bigl\{ z = x + iy : |x| \leq 1/2,\; 
        \bigl| y \pm x/\sqrt{D} \bigr| \leq \tfrac{D+1}{4\sqrt{D}} \Bigr\},
        \qquad D \in \{7,11\},
\]
which are depicted in \cref{fig:hexagonal_regions}.
\begin{rem}
To obtain a strict fundamental domain, we must account for boundary identifications under translation by $\mathcal{O}$. We therefore define
\begin{align*}
    \widetilde{I}_2 &= I_2 \;\setminus\; \bigcup_{\alpha \in \{1,\ \sqrt{-2}\}} (I_2 + \alpha), \\
    \widetilde{I}_D &= I_D \;\setminus\; \bigcup_{\alpha \in \left\{1,\ \frac{1\pm\sqrt{-D}}{2}\right\}} (I_D + \alpha), \qquad D\in \{7,11\}.
\end{align*}
\end{rem}
\noindent
Let  $D$ be fixed. Notice that for any $z\in \mathbb{C}$ there is a unique element $[z] \in \ringint$ such that $z-[z]\in \widetilde{I}_D$. We now define the following analogue of the Gauss map, called the \textit{Hurwitz map} (see \cite{ItoDed, Nakada, Kimetall}) as
\begin{equation}\label{def:Tgausmap}
    \Tgaus: I_D \to I_D,\qquad
    \Tgaus(z) :=
    \begin{cases}
        \displaystyle\frac{1}{z} - \boar{\frac{1}{z}}, & z \in I_D \setminus \{0\},\\
        0, & z = 0.
    \end{cases}
\end{equation}
Let us define the digit sets
\begin{align*}
    O_\alpha := \Bigl\{ z \in I_D : \boar{\frac{1}{z}} = \alpha \Bigr\}.
\end{align*}
The set above is empty only for units. As we are considering only $D=2,7,11$, this happens only for $\alpha=\pm 1$.

\begin{Def}[Continued fraction expansion in $\mathbb{Q}(\sqrt{-D})$]
Let $z \in I_D$. We call 
\[
\Bigl\langle \boar{\tfrac{1}{z}},\; \boar{\tfrac{1}{\Tgaus(z)}},\; \boar{\tfrac{1}{\Tgaus^2(z)}},\; \dots \Bigr\rangle
\]
the continued fraction expansion of $z$. The entries 
\[
\alpha_j = \boar{\tfrac{1}{\Tgaus^{\,j-1}(z)}} \in \mathcal{O}_K \setminus \{\pm 1\}
\]
are called the \textit{partial quotients} of $z$. Moreover, the expansion is finite if and only if $z \in I_D \cap \mathbb{Q}(\sqrt{-D})$.
\end{Def}

For each $z\in K$ we denote by $\ell(z)$ the exponent $j$ with $\Tgaus^j(z)\neq 0$ and $\Tgaus^{j+1}(z)=0$, and call it the length of the continued fraction expansion of $z$.

\begin{figure}[H]
    \centering
    \begin{subfigure}{0.30\linewidth}
        \centering
        \includegraphics[width=\linewidth]{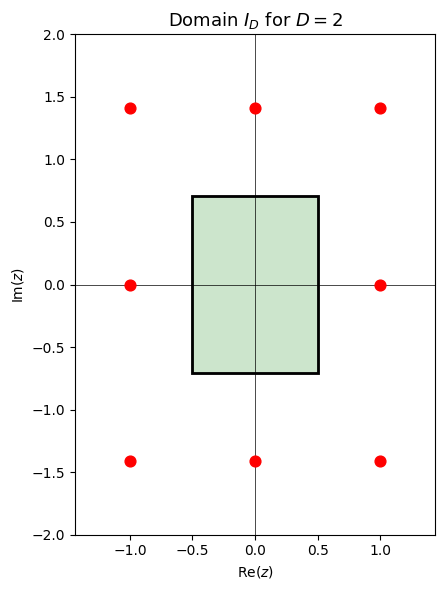}
        \label{fig:fundamental_region_d2}
    \end{subfigure}
    \hfill
    \begin{subfigure}{0.30\linewidth}
        \centering
        \includegraphics[width=\linewidth]{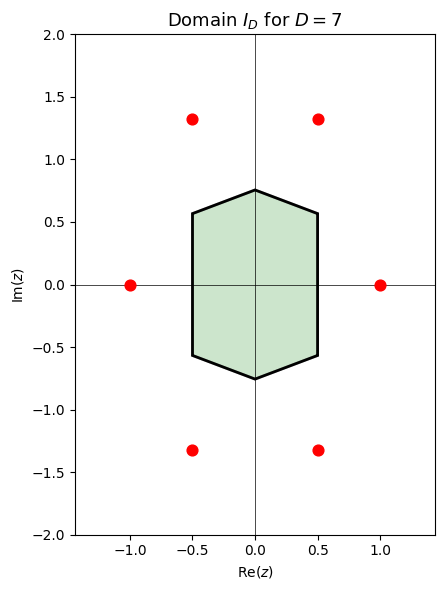}
        \label{fig:fundamental_region_d7}
    \end{subfigure}
    \hfill
    \begin{subfigure}{0.30\linewidth}
        \centering
        \includegraphics[width=\linewidth]{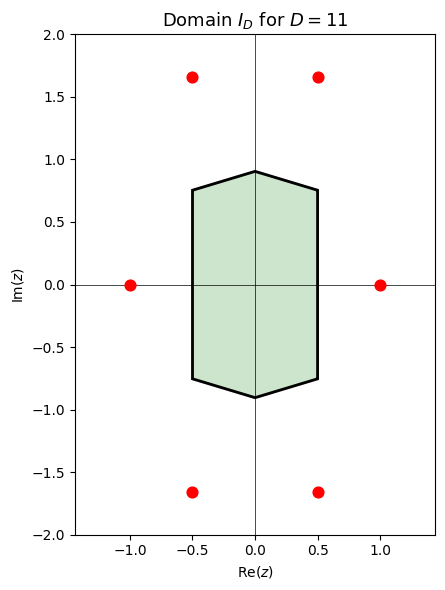}
        \label{fig:fundamental_region_d11}
    \end{subfigure}
    \caption{The fundamental domains $I_D$ for $D\in\{2,7,11\}$, 
    with algebraic integers marked in red (origin removed).}
    \label{fig:hexagonal_regions}
\end{figure}

\begin{rem}\label{rem:badalphas}
Plainly, the sets $O_\alpha$ above play the same role for the system $(I_D, \Tgaus)$ as the intervals $(\frac{1}{n+1}, \frac{1}{n}]$ do for $((0,1], \Cgaus)$. However, there are some important differences. For the Gauss map we have $\Cgaus((\frac{1}{n+1}, \frac{1}{n}])=(0,1]$ for all $n\geq 1$. On the other hand, as remarked in \cite{Kimetall}, there are $\alpha$ such that $\Tgaus(O_\alpha) \neq I_D$. These are precisely the $\alpha\in \ringint$ for which the translate $I_D+\alpha$ is not contained in the image of $h_\alpha: z\mapsto \frac{1}{z+\alpha}$ (see \cref{tab:nasty_algebraic_integers_D2} for $D=2$). Due to the shrinking property of $h_\alpha$ for $\alpha$ of large norm, the number of such elements is finite. This heavily affects the study of orbits of points $z\in I_D$ via dynamical methods. Indeed, a major difficulty that the authors of \cite{Kimetall} had to overcome was constructing a finite partition of $I_D$ and identifying an appropriate space of functions. We will return to this key point in the next sections.
\end{rem}

\begin{table}[H]
    \centering
    \setlength{\tabcolsep}{12pt}
    \renewcommand{\arraystretch}{1.5}
    \begin{tabular}{|c|c|}
        \hline
         $a$ & $b$ \\
        \hline
         $0$      & $\pm 1$ \\
         $\pm 1$  & $0,\ \pm 1$ \\
         $\pm 2$  & $0,\ \pm 1$ \\
        \hline
    \end{tabular}
    \caption{The algebraic integers $\alpha = a+ib\sqrt{2} \in \mathbb{Z}[\sqrt{-2}]$ 
    such that $f_\alpha(I_2)\not\subset I_2$ (for $D=2$).}
    \label{tab:nasty_algebraic_integers_D2}
\end{table}

\begin{Def}[Inverse branches]\label{defn:inversebranches}
For $\alpha \in \ringint\backslash \{\pm 1\}$ the maps $h_\alpha: \Tgaus O_\alpha\mapsto O_\alpha$ given by
\begin{align}
    h_\alpha(z):=\frac{1}{\alpha +z}
\end{align}
are called the inverse branches of $\Tgaus$. 
\end{Def}
In particular, for any non empty $O_\alpha$, the Hurwitz map $\Tgaus$ is a bijection from $O_\alpha$ to $\Tgaus O_\alpha$, given by $\Tgaus(z)=\frac{1}{z}-\alpha$. 
\begin{Def}[Depth of a branch, cf. \cite{Kimetall}]
Let $n\geq 1$, then given $\alpha_1, \dots, \alpha_n\in \ringint\backslash\{\pm 1\}$, define $h_{\tick{\alpha}}:=h_{\alpha_1}\circ \dots\circ  h_{\alpha_n}$. Define the sets $O_{\tick{\alpha}}=O_{(\alpha_1, \dots, \alpha_n)}$, inductively by
\begin{align*}
    &O_{(\alpha_1, \alpha_2)}=\{z\in O_{\alpha_1}: \Tgaus(z)\in O_{\alpha_2}\},\\
    &\vdots\\
    &O_{(\alpha_1, \dots, \alpha_n)}=\{z\in O_{\alpha_1}: \Tgaus(z)\in O_{(\alpha_2, \dots, \alpha_n)}\}.
\end{align*}
The map $h_{\tick{\alpha}}: \Tgaus^n O_{\tick{\alpha}}\mapsto O_{\tick{\alpha}}$ is a bijection. We call the index $n$ the depth of $\tick{\alpha}$ and write $\vert \tick{\alpha}\vert=n.$
\end{Def}
Let now $\tick{\alpha}=(\alpha_1, \dots, \alpha_n)$. Recall now that, for a function $f: \CC\to \CC$, we denote by $J_f$ the Jacobian determinant of $f$. The quantity
\begin{align}\label{eq:contractionratio}
    \rho:=\limsup_{n\to \infty} \sup_{\vert \tick{\alpha}\vert =n} \sup_{\substack{z\in \Tgaus O_{\tick{\alpha}}}} \vert J_{\tick{\alpha}}(z)\vert^{\frac{1}{n}}
\end{align}
is called \textit{the contraction ratio}, and is strictly smaller than 1 for all $D\in\{2,7, 11\}$, see \cite[Lemma 2.3]{Kimetall}.

\subsection*{Link with elliptic sums}
Let us consider two functions $\psi_1, \psi_2: I_D \to \mathbb{R}$. We extend these by periodicity modulo $2$, that is
\begin{align}\label{eq:psidefinitions}
\psi_j := \begin{cases} \psi_1 & \text{if } j \text{ is odd,} \\ \psi_2 & \text{if } j \text{ is even.} \end{cases}
\end{align}
For any given $z\in K\cap I_D$ define the quantity
\begin{align}\label{eq:total-cost}
    S(z):=\sum_{j=1}^{\ell(z)} \psi_j(\Tgaus^{j-1}(z))
\end{align}
and the \textit{period law}
\begin{align}\label{eq:periodeq}
\Psi(z) := \psi_1(z) + \psi_2\bigl(\Tgaus(z)\bigr).
\end{align}
We now proceed to link the above with our elliptic sums. Let $a_i \in \ringint$ and define a polynomial $P(a_0 , a_1 , \ldots , a_n)$ of $a_0 , a_1 , \ldots , a_n$ by the relations
\begin{align}
\begin{split}
    &P(a_0) = a_0, \\ 
    &P(a_0, a_1) = a_0 a_1 + 1,\\  
    &P(a_0, a_1, \ldots, a_m) = P(a_0, a_1, \ldots, a_{m-1})a_m + P(a_0, a_1, \ldots, a_{m-2}), \quad m \geq 2.
\end{split}
\label{eq:Pdef}
\end{align}
It is known (see \cite{ItoDed}) that
\begin{align*}
&P(a_0 , a_1 , \ldots , a_n) = a_0 P(a_1 , \ldots , a_n) + P(a_2 , \ldots , a_n), \\
&P(a_0 , a_1 , \ldots , a_n) = P(a_n , a_{n-1} , \ldots , a_0), \\
&P(a_0 , \ldots , a_n)P(a_1 , \ldots , a_{n-1}) - P(a_0 , \ldots , a_{n-1})P(a_1 , \ldots , a_n) = (-1)^{n+1}.
\end{align*}
Let now $a, c\in \ringint$ be two algebraic integers, $c\neq 0$. Then, since the field is Euclidean, we have
\[
\frac{a}{c} = \langle a_0, \dots, a_n \rangle = \frac{P(a_0, \ldots , a_n)}{P(a_1, \ldots , a_n)}.
\]
The following lemma is an analogue of the Barkan–Hickerson expression for Dedekind sums (see e.g. \cite[Lemma 4]{Conrey}).

\begin{lem}\label{lem:Ito-Hickerson}
Let $\frac{a}{c}=\langle a_0 , a_1 , \ldots , a_n \rangle$ for $a_i\in \ringint$. Then
\begin{multline*}
\mathfrak{D}(a, c) = G_2(0) \, I\Bigg( \frac{P(0, a_1 , \ldots , a_n)}{P(a_1 , \ldots , a_n)} + (-1)^{n+1} \frac{P(0, a_n , \ldots , a_1)}{P(a_n , \ldots , a_1)} \\
+ a_1 - a_2 + \cdots + (-1)^{n+1} a_n \Bigg),
\end{multline*}
where $I(z)=2i\Im(z)$ and $G_2(0)\not =0$ is the valued prescribed by the analytic continuation of the functions $G_2(z)$ in \eqref{eq:Gkintro}.
\end{lem}

\begin{proof}
This is \cite[Theorem 3]{ItoDed}, where we dropped the condition $P(a_m, \dots, a_n)\not =0$ for all $m\leq n$, as this is automatic in our case, see \cite[pp. 202-203]{ItoDed}.
\end{proof}

\noindent
Due to the above lemma, we see that for any coprime pair, the normalized sum $\tilde{\mathfrak{D}}(a,c)$ given by
\begin{align*}
    \tilde{\mathfrak{D}}(a,c):=\frac{1}{i\sqrt{|D|}\, G_2(0)} \mathfrak{D}(a,c),
\end{align*}
is real. Specify
\[
\psi_1(z) = \frac{2}{\sqrt{D}}\Im\left(\boar{\frac{1}{z}}\right) 
\qquad\text{and}\qquad 
\psi_2(z) = -\frac{2}{\sqrt{D}}\Im\left(\boar{\frac{1}{z}}\right),
\]
and observe that 
\begin{align}\label{eq:dedekindapprox}
    S\left(\frac{a}{c}\right)= \tilde{\mathfrak{D}}\left(\frac{a}{c}\right)+\LandauO(1).
\end{align}
Therefore, we are interested only in the period function 
\begin{align}\label{eq:eqperiod}
\Psi(z)=\frac{2}{\sqrt{|D|}}\Im\left(\boar{\frac{1}{z}}-\boar{\frac{1}{\Tgaus(z)}}\right).
\end{align}
We now equip the set $\FFFF(X)$ with the uniform probability measure, and denote by $\ExpX$ the corresponding expectation. 
Let $t$ be a real variable. The characteristic function of the map $S: I_D \to \mathbb{R}$ with $z\mapsto S(z)$ is given by
\begin{align}\label{eq:charfunctions}
    \chi_{S, X}(t)=\chi_S(t) := \ExpX[e^{itS(z)}]=\frac{1}{\# \FFFF(X)} \sum_{z \in \FFFF(X)} e^{itS(z)}.
\end{align}
Our goal is to obtain an asymptotic evaluation of the above characteristic function as $|t| \to 0$.

\begin{rem}
In order to simplify the exposition, we restrict ourselves to the explicit map in \eqref{eq:eqperiod}. However, most of the discussion below works \textit{mutatis mutandis} for a wider class of functions. For example, $\psi_i(z)$ can be taken to be of the form $\Re\!\left(\boar{\frac{1}{z}}\right)^c$, or $\Im\!\left(\boar{\frac{1}{z}}\right)^c$ for some $c>0$. The value of $c$ influences the distribution sharply: as it will be clear in the forcoming sections, for $c>1$ one must expect some non Gaussian limiting distribution, while for $c<1$ one expect a Gaussian behavior.
\end{rem}

\section{Preparatory background }\label{sec:prep-material}
\subsection{A Markov partition for \texorpdfstring{$(I_D, \Tgaus)$}{(I\_D, T)}}
A major complication in studying the dynamics for the system $(I_D, \Tgaus)$ -carefully highlighted in \cite{Kimetall}- is that the map $\Tgaus$ is not always a full branch map, that is, we have \textit{some} $\alpha \in \ringint \backslash \{\pm 1 \}$ for which $\Tgaus O_{\alpha} $ is a proper subset of the interior of $\tilde{I}_D$. For the case $D=2$, those are listed in \cref{tab:nasty_algebraic_integers_D2}.
Concretely, this is of fundamental importance for a proper definition of a the density transformer and its spectral analysis. In the sequel, following \cite{Kimetall}, we will define the operator $\Myoperator_{s,t}$ for $(s, t)\in \CC\times \RR$ over a certain Banach space of functions $\mathscr{C}(\mathscr{P})$ by 
\begin{align*}
    \Myoperator_{s,t} f(z)
    := \sum_{z_0 \in \Tgaus^{-1}(z)} g_{s,t}(z_0) f(z_0)
    = \sum_{\alpha \in \ringint} (g_{s,t} f) \circ h_\alpha(z) \,
    1_{\Tgaus \mathcal{O}_\alpha}(z),
\end{align*}
where $g_{s,t}(z)$ is a factor that will be specified later.

To address the difficulties mentioned above, we draw from the work of Ei, Nakada and Natsui \cite{Nakada} and \cite{Kimetall} (which relies on \cite{Nakada}), where a finite Markov partition $\Mpartition$ of the set $I_D$ is constructed. We summarize the key ingredients of this construction below; the partition consists of singletons, line segments, circular arcs, and regions bounded by such, as illustrated in \cref{fig:partition-unlabeled}.

\begin{prop}[Markov partition for $I_D$]\label{prop:compatiblepartition}
For $D\in \{2,7,11\}$ there is a finite collection $\Mpartition$ of subsets $P$ of $I_D$ with $I_D = \bigcup_{P\in \Mpartition} P$ and the following properties:
\begin{enumerate}
    \item For any $\alpha\in \ringint \setminus \{\pm 1\}$, $O_\alpha$ is a disjoint union of elements of $\Mpartition$.
    \item For each inverse branch $h_\alpha$ and each $P\in \Mpartition$, exactly one of the following holds:
    \begin{itemize}
        \item There is a unique $Q = Q(P) \in \Mpartition$ such that $h_\alpha(P) \subset Q$;
        \item $h_\alpha(P) \cap I_D = \emptyset$.
    \end{itemize}
\end{enumerate}
Furthermore, the partition $\mathscr{P}$ can be split into the union $\bigcup_{i=0}^2 \mathscr{P}[i]$ of its $i$-dimensional parts for $i\in \{0,1,2\}$.
\end{prop}

\begin{proof}
This is essentially Proposition 3.7 in \cite{Kimetall}; see also the discussion in Section 3 therein.
\end{proof}

\begin{figure}[H]
    \centering
    \includegraphics[width=1\linewidth]{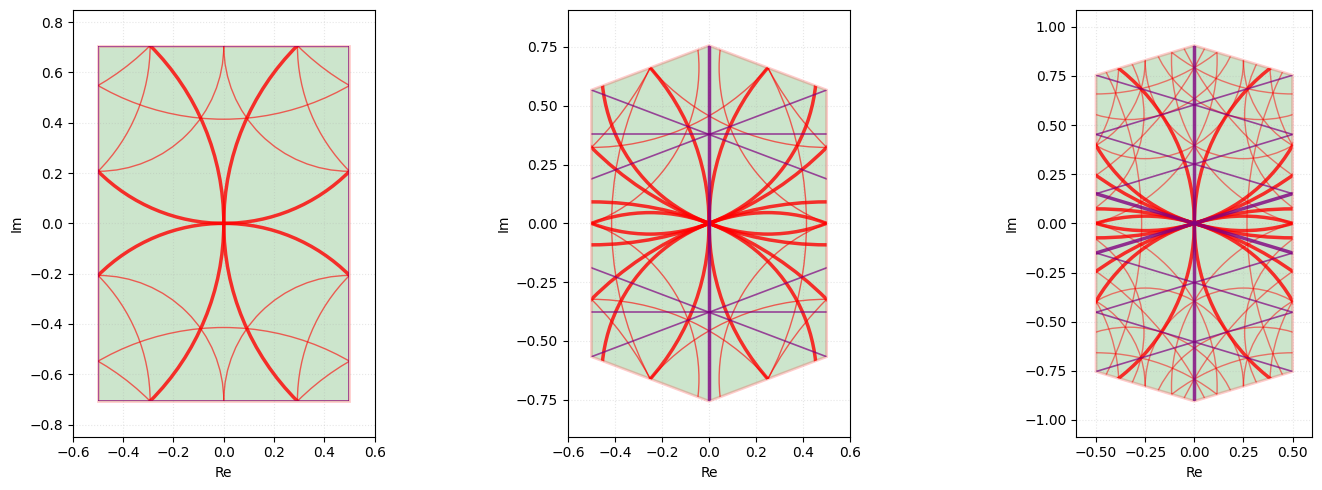}
    \caption{The partition $\Mpartition$ for $D=2,7,11$. The parts in $\Mpartition[2]$ are exactly the open subsets bounded by red arcs or purple lines. The plot was realized using the sets of lines and circles explicitly given in \cite{Nakada}}
    \label{fig:partition-unlabeled}
\end{figure}

In view of the above partition, the operator we are going to define will require, for computations, a certain understanding of the restrictions of the inverse branches to the various parts of $\mathscr{P}$. In what follows, we follow closely \cite[Section 2--4]{Kimetall}.

\medskip

By the proposition above, for any eligible $\alpha\in \ringint$ and any $P\in \mathscr{P}$, there is a unique $Q$ in $\mathscr{P}$ with $h_\alpha(P) \subset Q$ (whenever $h_\alpha(P) \cap I_D \neq \emptyset$). We denote the restriction of $h_\alpha$ to $P$ by $\langle\alpha\rangle_Q^P$. Now, for any pair of sets $P,Q \subset \mathscr{P}$ we define the set of maps 
\begin{align}
    B(P,Q):= \{ h: P \to Q \mid h = \langle \alpha\rangle_Q^P,\ \alpha \in \ringint \setminus \{\pm 1\} \}.
\end{align}
More generally, given $\tick{\alpha}=(\alpha_1, \dots, \alpha_n)$ and a part $P\in \Mpartition$, we consider the restriction of
\[
h_{\tick{\alpha}} = h_{\alpha_1} \circ h_{\alpha_2} \circ \cdots \circ h_{\alpha_n}
\]
to $P$ as 
\begin{align*}
    \langle \tick{\alpha}\rangle := h_{\tick{\alpha}}\big|_P.
\end{align*}
We denote the set of (local) inverse branches of depth $n$ from $P$ to $Q$ for $P, Q \in \Mpartition$ by
\begin{align}\label{eq:resbranch:dep:n}
   B^n(P,Q) := \{ \text{all maps of the form } \alphares \text{ of depth } n \text{ from } P \text{ to } Q \}.
\end{align}
We also define the sets 
\begin{align*}
    B^\star(P,Q) := \bigcup_{n\geq 1} B^n(P,Q), \qquad 
    B^\star := \bigcup_{Q,P\in \Mpartition} B^\star(P,Q).
\end{align*}

\subsection{The space of functions and its norms}
We now describe the space of functions on which the transfer operator will act.
\begin{Def}[Cf. \cite{Kimetall}]\label{def:spaceoffun}
    Let $\mathscr{C}(\mathscr{P})$ be the set of functions from $I_D$ to $\CC$ such that for any $P\in \mathscr{P}$, the function $f\times 1_P$ extends holomorphically to some open neighborhood U containing $\bar{P}$. 
\end{Def}

The space above has a nicer representation  as follows. We define the restriction map $res_\mathscr{P}$ as follows. First, for every $P\in \mathscr{P}$ we define $res_P(f)$ to be the restriction to $\bar{P}$ of the extension of $(f\times 1_P)$ in an open neighborhood of $P$. The latter is well defined by uniqueness of such an extension and defines an element in $\mathscr{C}(\bar{P})$. The map $res_\mathscr{P}$ is then defined by
\begin{equation}\label{split-space}
    \operatorname{res}_{\mathscr{P}}: \mathscr{C}(\mathscr{P}) \longrightarrow \bigoplus_{P \in \mathscr{P}} \mathscr{C}(\bar{P}), \qquad
    f \longmapsto \bigl( \operatorname{res}_P(f) \bigr)_{P \in \mathscr{P}}.
\end{equation}
which is a bijection, see \cite[Proposition 4.2]{Kimetall}.
\newline
Recall the decomposition $\Mpartition=\cup_{i\in {0,1,2}}\Mpartition[i]$, where $\Mpartition[i]$ is the collection of all open $i$-dimensional cells. Following \cite{Kimetall} we introduce some norms and semi-norms. Given any $P\in \mathscr{P}$, for $f\in \mathscr{C}(\bar{P})$ set
\begin{align}
    \Vert f_P\Vert_a:=\sup_{z\in P} \vert f(z)\vert.
\end{align}
Moreover, if $P\not \in \mathscr{P}[0]$, define the semi-norm
\begin{align}
    \Vert f_P\Vert_b=\sup_{z\in P} \sup_{v} \vert \partial_v f(z)\vert,
\end{align}
where the innermost supremum being taken over the set of all unit tangent vectors for the directional derivative. We set $\Vert f_P\Vert_b=0$ if $P$ is a singleton. For any parameter $t>0$ consider
\begin{align*}
    \Vert f_P \Vert_{c(t)}=\Vert f_P\Vert_c :=\Vert f_P\Vert_a+ \frac{1}{t} \Vert f_P\Vert_b,
\end{align*}
where we will sometimes omit the dependency on $t$.
\noindent To define a norm over the space $\mathscr{C}(\mathscr{P})$, define
\[
\Vert f\Vert_a:=\max_{P\in \Mpartition} \Vert f_P\Vert_a, \quad 
\Vert f\Vert_{b}:=\max_{P\in \Mpartition} \Vert f_P\Vert_b, \quad 
\Vert f\Vert_{c}:=\Vert f\Vert_a+\frac{1}{t}\Vert f\Vert_b.
\]

Finally, given any linear operator $A: \mathscr{C}(\mathscr{P}) \to \mathscr{C}(\mathscr{P})$ we will abuse notation and denote
\begin{align*}
\Vert A\Vert_\star=\sup_{\Vert f\Vert_\star=1} \Vert A f\Vert_\star,
\end{align*}
where $\star\in \{a,b,c\}$. It is possible to show that for any nonzero $t\not =0$ all $c-$norms are equivalent. Finally, we have the following proposition, ensuring we are dealing with a Banach space.
\begin{prop}
    The space $(\mathscr{C}(\mathscr{P}), \Vert \Vert_{c(1)})$ is a Banach space.
\end{prop}
\begin{proof}
    This is \cite[Proposition 4.4]{Kimetall}.
\end{proof}

\section{Transfer operators, basic norm estimates and link to generating series}\label{sec:transoperator}
We proceed now in defining an appropriate two parameter family of operators acting on $\mathscr{C}(\mathscr{P})$. At this point, our analysis partially diverges from the approach used in \cite{Kimetall} in two directions. First, since we are not relying on Hwang quasi-power theorem, we will restrict to a real variable $t$. Secondly, our operator $\Aopera_{s,t}$ (see below) is a composition of two other operators, as in \cite{Bettin-Drappeau-integral}. 
\par
Let now $(s,t)\in \CC \times \RR$. For $z\in P$ and $P\in \Mpartition$ and $j=\{1,2\}$ set
\begin{align*}
    g_{s,t,j}(z):=e^{it \psi_j (z)} J_{[1/z]} (\Tgaus(z))^{s},
\end{align*}
where $\psi_j$ are defined as in \eqref{eq:psidefinitions} and $J_\alpha(z)$ denotes the determinant of the Jacobian of $h_\alpha(z)$:
\[
J_\alpha(z)=\vert h'_\alpha(z)\vert^2=\frac{1}{\vert z+\alpha\vert^4}>0,
\]
which is always nonzero.
Then we define the following operators 
\begin{align}
 \Myoperator_{s,t}^{(j)}[f](z)=\sum_{\alpha \in \ringint} ( g_{s,t, j} f)\circ h_\alpha (z) \cdot 1_{\Tgaus O_\alpha}(z).
 \end{align}
 Recall that by \cref{prop:compatiblepartition} and \cref{def:spaceoffun}, the characteristic function of $\Tgaus O_\alpha$ is in $\mathscr{C}(\mathscr{P})$ for any $\alpha \in \ringint\backslash\{\pm 1\}$. Then, for $z\in P$, these operators have the representation
\begin{align}\label{eq:basicop-rep}
    (\Myoperator_{s,t}^{(j)}[f](z))_P=\sum_{Q\in \Mpartition} \sum_{\alphares \in B(P, Q)} e^{it(\psi_j \circ \alphares(z))} J_{\tick{\alpha}}(z)^s (f_Q \circ \alphares (z)),
\end{align}
for $j\in\{1,2\}.$ We then define the composition of the two operators above:
\begin{align*}
    \Aopera_{s,t}[f](z):=\left(\Myoperator_{s,t}^{(2)} \circ \Myoperator_{s,t}^{(1)}\right)[f](z).
\end{align*}
Specify the following operator
\begin{align*}
\Myoperator[f](x):=\Myoperator_{1,0}[f](x),
\end{align*}
which will play a crucial role in the subsequent chapters. Notice that any fixed $P\in \Mpartition$ and $z\in P$, we have the explicit representation
\begin{align*}
    \Myoperator_{1,0}[f](z):=\sum_{Q\in \Mpartition} \sum_{\alphares \in B(P, Q)} \vert h_\alpha'(z)\vert^2 (f_Q \circ \alphares)(z).
\end{align*}

 \begin{prop}
   Let $\Re(s)>1/2$, $ t\in \RR$ and let $n\geq 1$ be an integer. Then the operators $\Myoperator_{s,t}^{(j)}$ and $\Aopera_{s,t}$ are well defined linear operators acting on $\mathscr{C}(\mathscr{P})$. Moreover, for any fixed $P\in \Mpartition$ and any $z\in P$, we have the explicit representation of the powers of the operator $\Myoperator$ as 
\begin{align*}
    (\Myoperator^n [f](z))_P=\sum_{Q} \sum_{\alphares \in B^n(P,Q)} {J_{\tick{\alpha}}(z)} (f_Q \circ \alphares  (z)).
\end{align*}
\end{prop}

\begin{proof}
Given a fixed $P \in \Mpartition$, to prove the convergence of the series in \eqref{eq:basicop-rep} notice that for any $t$ and any $s$ with $\Re(s)>1/2$ we have
\begin{align*}
      \sum_{Q\in \Mpartition} \sum_{\alphares \in B(P, Q)}\vert h_\alpha'(z)\vert^{2s}\ll \sum_{\alpha \in \ringint} \frac{1}{\vert z+\alpha\vert^{4s}},
\end{align*}
which is finite. Since the partition has only finitely many parts, this holds globally. Let  $P\in \Mpartition$ be a part, and let $\alphares\in B(P,Q)$. For any $Q\in \Mpartition$, one has $f_Q\in \mathscr{C}(\mathscr{P})$. 
For $j=1$ and any $\alphares\in B(P,Q)$ we have 
\[
\exp\left(it(\psi_j\circ \alphares (z))\right)=\exp\left(it \frac{2}{\sqrt{D}}\Im\left(\alpha\right)\right).
\]
This follows by the fact that, for a depth 1 inverse branch $h_\alpha$, we have $\boar{\frac{1}{h_\alpha (z)}}=\alpha$ for any $z\in I_D$. Whence, the exponential factor does not depend upon $z$, but only on $\alphares$. The same holds also for $j=2$. Therefore, these operators are well defined. 
The last assertion follows at once by standard computations and the definition of $B^n(P,Q)$ (in particular, this follows immediately if one notices that our $\Myoperator$ matches the operator $\Myoperator_{s, t}$ with $s=1, w=0$ in \cite{Kimetall}).
\end{proof}

Our next goal is to express the two parameter family of operators  $\Aopera_{s,t}$ as a perturbation of the operator $\Myoperator$. This is a standard practice in the application of the thermodinamical formalism to distributional problems (see \cite{Broise}) and will be crucial in our analysis. Define the factor
\begin{align}\label{eq:gst}
    g_{s,t}(z):=e^{it\Psi(z)} T^{s-1}(z),
\end{align}
where $\Psi$ is as in \eqref{eq:eqperiod} and $T$ is defined as $T(z)=\vert z\Tgaus(z)\vert^4$.

\begin{prop}\label{prop:reduction}
For $\Re(s)>\frac{1}{2}$ and $t\in \RR$, the operator $\Aopera_{s,t}$ can be rewritten as $f\mapsto \Myoperator^2[g_{s,t} f]$.
\end{prop}

\begin{proof}
    Let $f \in \BanachHolo$ and let $z \in P$ for some $P \in \mathscr{P}$. By definition we have
    \begin{align*}
        \Myoperator_{s,t}^{(j)}[f](z) = \sum_{Q \in \Mpartition} \sum_{\alphares \in B(P, Q)} 
        e^{it(\psi_j \circ \alphares(z))} \, 
        \vert h_{\tick{\alpha}}'(z) \vert^{2s} \, 
        (f_Q \circ \alphares(z)).
    \end{align*}
    Since for $\Re(s) > 1/2$ the sum converges absolutely, we compute
    \begin{align*}
        \Aopera_{s,t}[f](z)
        &= \sum_{Q', Q \in \Mpartition} \,
        \sum_{\substack{\langle \tick{\alpha_1}\rangle \in B(P,Q') \\ 
                        \langle \tick{\alpha_2}\rangle \in B(Q', Q)}}  
        (\star \star) \, 
        \vert h'_{\tick{\alpha_1}}(h_{\tick{\alpha_2}}(z)) \vert^{2s} \,
        \vert h_{\tick{\alpha_2}}'(z) \vert^{2s} \,
        \bigl(f_{Q'} \circ \langle \tick{\alpha_1}\rangle \circ \langle \tick{\alpha_2}\rangle(z)\bigr),
    \end{align*}
    where 
    \[
        (\star \star) = \exp\Bigl(it\bigl[(\psi_2 \circ \tick{\alpha_2})(z)  
                      + (\psi_1 \circ \tick{\alpha_1} \circ \tick{\alpha_2})(z)\bigr]\Bigr).
    \]
    Write $\tick{\alpha} = (\alpha_1, \alpha_2)$, so that $\alphares = (h_{\alpha_1} \circ h_{\alpha_2})_P$.
    Notice that $\boar{\frac{1}{\bullet}} \circ h_{\alpha_2}(z) = \alpha_2$ trivially. 
    On the other hand, for any $z \in P$, we have $h_{\alpha_2}(z) \in Q' \subset I_D$. 
    Therefore, $\boar{\frac{1}{\bullet}} \circ h_{\alpha_1} \circ h_{\alpha_2}(z) = \alpha_1$ 
    for any $z \in P$. Now, using the chain rule, the definition of $B^n(P, Q)$ \eqref{eq:periodeq} 
    and \eqref{eq:eqperiod}, we have
    \begin{align*}
        \bigl(\Aopera_{s,t}[f](z)\bigr)_P
        = \sum_{Q \in \Mpartition} \sum_{\alphares \in B^2(P,Q)}  
        e^{it \Psi \circ \alphares (z)} \, 
        \vert h_{\tick{\alpha}}'(z) \vert^{2s} \,
        \bigl(f_Q \circ \tick{\alpha}(z)\bigr),
    \end{align*}
    which coincides with $\Myoperator^2[g_{s,t} f]$, since $\tick{\alpha} = (\alpha_1, \alpha_2)$ and 
    \[
    T\circ \alphares (z)=\vert \alpha_1\circ\alpha_2 (z)\vert^4 \vert \alpha_2 (z)\vert^4= J_{\alpha_1}(\alpha_2(z))J_{\alpha_2}(z).
    \]
\end{proof}

\begin{rem}\label{rem:quicknote}
    We remark here a subtle issue. In the above we used the fact that 
    \[
        \boar{\frac{1}{\bullet}} \circ h_{\alpha_1} \circ h_{\alpha_2}(z) = \alpha_1,
    \]
    for $z \in P$. This does \textit{not} hold true for all $\alpha$ if we do not restrict 
    the inverse branch $h_{\alpha_2}$ to $P$. Indeed, there are finitely many $\alpha \in \ringint$, like those in \cref{tab:nasty_algebraic_integers_D2}, such that 
    $h_{\alpha}(I) \not\subset I_D$, and therefore, it can happen that for some $\beta \neq \alpha_1$ 
    we have $\bigl[\frac{1}{\bullet}\bigr] \circ h_{\alpha_1} \circ h_{\alpha_2}(z) = \beta$.
    
    We also point out a key feature that will significantly ease our exposition later: 
    given any inverse branch $\alphares = (h_{(\alpha_1, \alpha_2)})_P \in B^2(P, Q)$, the composition 
    \[
        \exp\!\bigl(it (\Psi \circ h_{\tick{\alpha}})(z)\bigr)
        = \exp\!\Bigl( it \bigl[ (\psi_1 \circ h_{\tick{\alpha}})(z)
                               + (\psi_2 \circ \Tgaus \circ h_{\tick{\alpha}})(z) \bigr] \Bigr)
        = \exp\!\Bigl( it \,\frac{2}{\sqrt{D} } \, 
                      \Im(\alpha_1 - \alpha_2) \Bigr)
    \]
    depends only on $\tick{\alpha}$, not on $z \in P$.
\end{rem}

As we will see in \cref{sect:spectral-theory}, the point $\Aopera_{1,0}$ exhibits special spectral properties (cf. \cref{prop:quasi-comp-swfamily}), and the same holds for all points in a sufficiently small neighborhood of $\Aopera_{1,0}$. The next lemma provides continuous bounds for the two-parameter family $\Aopera_{s,t}$. This result is a straightforward adaptation of \cite[Lemma 5.3]{BettinDrappeau-main} to our setting. However, given its frequent use in the subsequent chapters, we include a sketch of its proof for completeness.

\begin{lem}[Continuous bounds for the perturbation]\label{lem:3-step-lemma}
    Let $\delta$ be sufficiently small. There is a sufficiently small $t_0 > 0$ such that the following estimates hold:
    \begin{itemize}
        \item For $\sigma \geq 1-\delta$ and $|t| \leq t_0$, we have
        \begin{align}\label{eq:contbound-first}
            \Vert \Aopera_{s,t}-\Aopera_{s, 0}\Vert_a \ll_\delta |t|.
        \end{align}
        \item For $1-\delta < \sigma \leq 1$ and $\tau \in \mathbb{R}$, we have
        \begin{align}\label{eq:contbound-second}
            \Vert \Aopera_{\sigma+i\tau, 0}-\Aopera_{1+i\tau, 0}\Vert_a \ll |\sigma-1|.
        \end{align}
        \item Given $\tau_1, \tau_2 \in \mathbb{R}$, we have
        \begin{align}\label{eq:contbound-third}
            \Vert \Aopera_{1+i\tau_1, 0}-\Aopera_{1+i\tau_2, 0}\Vert_a \ll |\tau_1-\tau_2|.
        \end{align}
    \end{itemize}
\end{lem}

\begin{proof}
Let $z \in P$ with $P \in \Mpartition$. By the bounded distortion property (cf. \cite[Proposition 2.4 and Proposition 3.10]{Kimetall}), for any $z \in I_D$ we have $|J_{\tick{\alpha}}(0)| \asymp |J_{\tick{\alpha}}(z)| = |J_{\alpha_1}(0)| \, |J_{\alpha_2}(0)|$ uniformly on all $\tick{\alpha}=(\alpha_1, \alpha_2) \in B^2(P, Q)$ and all $P,Q$. Therefore, we estimate
\begin{align*}
    \Vert \Aopera_{s,t}-\Aopera_{s, 0}\Vert_a 
    &\ll \sup_{\Vert f \Vert_a=1} \frac{\Vert \Myoperator^2[(g_{s,t}-g_{s,0}) f] \Vert_a}{\Vert f\Vert_a} \\
    &\ll \sum_{Q \in \Mpartition} \sum_{{\tick{\alpha}} \in B^2(P,Q)} 
    |J_{{\tick{\alpha}}}(0)|^\sigma 
    \| \exp(it \Psi(\tick{\alpha}(z))) - 1 \|_a.
\end{align*}
Since $z$ is contained in the bounded domain $I_D$, we have 
\[
\exp(it \Psi(\tick{\alpha}(z))) - 1 \ll |t| \sup_{z \in P} |\Psi \circ \tick{\alpha} (z)\vert.
\] 
Therefore
\begin{align*}
    \Vert \Aopera_{s,t}-\Aopera_{s, 0}\Vert_a 
    &\ll |t| \sum_{ Q \in \Mpartition} \sum_{{\tick{\alpha}} \in B^2(P,Q)} 
    |J_{\tick{\alpha}}(0)|^\sigma
    \sup_{z \in P} |\Psi \circ \tick{\alpha} (z)|.
\end{align*}
From the definition of $\Psi$ we see that 
\begin{align*}
    \| \Psi|_{\tick{\alpha}(P)} \|_a 
    &\leq \| \psi_1|_{\alpha(P)} \|_a 
    + \| (\psi_2 \circ \Tgaus)|_{\alpha(P)} \|_a \\
    &\leq |\alpha_1| + |\alpha_2|.
\end{align*}
Therefore, the innermost sum is 
\begin{align*}
    &\ll \sum_{ \tick{\alpha_1}, \tick{\alpha_2} \in B(P,Q)} 
    |J_{\tick{\alpha_1}}(0) J_{\tick{\alpha_2}}(0)|^\sigma
    \bigl( |\alpha_1| + |\alpha_2| \bigr) 
    \ll \sum_{\alpha' \in \ringint} \frac{1}{|\alpha'|^{3\sigma}}.
\end{align*}
This proves \eqref{eq:contbound-first}. Turning to \eqref{eq:contbound-second}, we proceed similarly and obtain analogous bounds. We may assume without loss of generality that $\sigma \leq 1$.
\begin{align*}
    \Vert \Aopera_{\sigma+i\tau,0}-\Aopera_{1+i\tau, 0}\Vert_a
    &\ll \sum_{Q \in \Mpartition} \sum_{{\tick{\alpha}} \in B^2(P,Q)} 
    J_{\tick{\alpha}}(0) \| (g_{\sigma+i\tau, 0}-g_{1+i\tau, 0}) \circ \tick{\alpha} (z) \|_a \\
    &\ll |\sigma-1| 
    \sum_{Q \in \Mpartition} \sum_{{\tick{\alpha}} \in B^2(P,Q)} 
    J_{\tick{\alpha}}(0) \bigl( 1+ \log J_{\tick{\alpha}}(0) \bigr) \\
    &\ll |\sigma-1| 
    \sum_{\alpha_1, \alpha_2 \in \ringint} 
    \frac{1}{(\alpha_1\alpha_2)^{5}} 
    \bigl( 1 + \log|\alpha_1\alpha_2| \bigr).
\end{align*}
Here we used the bounded distortion property again. The last expression converges. Finally, for \eqref{eq:contbound-third}, restricting to $P$, we have:
\begin{align*}
    \Vert \Aopera_{1+i\tau_1,0}-\Aopera_{1+i\tau_2,0}\Vert_a 
    &\ll \sum_{Q \in \Mpartition} \sum_{{\tick{\alpha}} \in B^2(P,Q)} 
    J_{\tick{\alpha}}(0) 
    \| J_{\tick{\alpha}}(z)^{i\tau_1} - J_{\tick{\alpha}}(z)^{i\tau_2} \|_a\\
    &= \sum_{P \in \Mpartition} \sum_{{\tick{\alpha}} \in B^2(P,Q)} 
    J_{\tick{\alpha}}(0) 
    \| J_{\tick{\alpha}}(z)^{i(\tau_1-\tau_2)} - 1 \|_a\\
    &\ll |\tau_1-\tau_2| 
    \sum_{P \in \Mpartition} \sum_{{\tick{\alpha}} \in B^2(P,Q)} 
    J_{\tick{\alpha}}(0) \bigl( 1 + \log |J_{\tick{\alpha}}(0)| \bigr)\\
    &\ll |\tau_1-\tau_2|,
\end{align*}
where we again used the bounded distortion property.
\end{proof}

\subsection{The generating series}
Let now $v\in K \cap I_D$. We have $v=\frac{a}{b}$ for some $a, b\in \ringint$ and we set an height on $K$ by $\hfat: K\to \RR$ by $\hfat(0)=0$ and $\hfat(v)=\vert b\vert $ otherwise. We then form the series
\begin{align}
    H(s,t)=\sum_{\substack{v\in K}} \frac{1}{\hfat(v) ^{2s}} e^{it S(v)}.
\end{align}
Now, for any $v=\frac{a}{b}\in K$, we have $v=\langle 0, \alpha_1, \dots, \alpha_{\ell(v)}\rangle$ for some $\alpha_i\in \ringint$. A simple computation shows that  
\begin{align}
    v=h_{\alpha_1} \circ \dots \circ h_{\ell(v)} (0)= h_{{\tick{\alpha}}}(0),
\end{align}
with $\tick{\alpha}=(\alpha_1, \dots, \alpha_{\ell(v)})$. The following lemma relates the height of $v$ to the Jacobian of the (restriction of) the branch induced by $\tick{\alpha}$. We recall that the origin 0 belongs to $\Mpartition[0]$, and is therefore a legitimate cell of the partition (see also \cite[Section 8]{Kimetall}).
\begin{lem}
    Let ${\tick{\alpha}}\in B^n(0, Q)$ and let $v=h_{\tick{\alpha}} (0)$ and $v=\frac{a}{b}$. Then we have $J_{\tick{\alpha}}(0)^s=\vert b\vert^{4s}$.
\end{lem}

\begin{proof}
    See \cite[Lemma 8.1]{Kimetall}.
\end{proof}
For $\Re(s)>1$, define the series
\begin{align}
\left(\myid+\Myoperator_{s, t}^{(1)}\right)\left( \myid-\Aopera_{s,t}\right)^{-1}[1](0),
\end{align}
where $1$ here denotes the indicator function of $I_D$, which is in $\BanachHolo$. By the same computations as in \cite{BettinDrappeau-main, Kimetall, Morris, BV-Euclidean} we find that for $\Re(s)>1$, we have the equality
\begin{align}\label{eq:gen-series-link}
\left(\myid+\Myoperator_{s,t}^{(1)}\right)\left( \myid-\Aopera_{s,t}\right)^{-1}[1](0)= H(2s, t)=\sum_{n\geq 1} \frac{a_{t,n}}{n^{2s}},
\end{align}
where
\begin{align*}
    a_{n, t}:=\sum_{\substack{v\in \FFFF\\ \hfat(v)^2=n}} e^{it S(v)}.
\end{align*}

\begin{rem}
    The Dirichlet series in \eqref{eq:gen-series-link} converges absolutely for $\Re(s)>1$. This can be seen by noticing that $\vert a_{n, w}\vert \ll n \times r_D(n)$, where $r_D(n)$ is the number of representations of $n$ as $a^2+Db^2$, where here (depending of $D$) $a,b$ can be half integers. So we get $\vert a_{n,w}\vert\ll n^{1+\epsilon}$. 
\end{rem}

\section{Meromorphic continuation for the generating series}\label{sect:spectral-theory}

In this section, as is customary in this type of approach (see \cite{BV-Euclidean, BettinDrappeau-main, Kimetall}), our goal is to obtain the analytic continuation of $H(2s,t)$ as well as sufficient decay in $\tau$, where $s = \sigma + i\tau$, in a strip $\sigma > 1 - \delta$ for some $\delta > 0$. The argument splits into two main parts.
The first part establishes the continuation for small values of $|\tau|$ and any $|t| \leq t_0$ (for some $t_0 > 0$). This requires only perturbation theory of operators and the continuity of the family $\Aopera_{s,t}$ with respect to the norms introduced previously.
The second part, where a pointwise estimate for large values of $\tau$ is established, requires more work. In particular, one employs ideas of Dolgopyat \cite{Dolgopyat}, which were subsequently adapted to the number-theoretic setting by Baladi and Vallée \cite{BV-Euclidean}.
\\
\par
We recall that for the space $\BanachHolo$, one has the decomposition (see \eqref{split-space})
\[
\BanachHolo=\mathscr{C}(\Mpartition[0])\cup \mathscr{C}(\Mpartition[1])\cup \mathscr{C}(\Mpartition[2]).
\]
We begin with the following lemma, which  states that the operator $\Aopera_{s,t}$ is quasi-compact for $(s,t)$ in a neighborhood of $(1,0)$.
\begin{prop}[A Ruelle-Perron-Frobenius theorem for $\Aopera_{s,t}$]\label{prop:quasi-comp-swfamily}
There exist parameters $\delta, \tau_0, t_0>0$ such that, for all $(s,t)$ satisfying
$\lvert 1-\sigma\rvert \leq \delta$, $\lvert \tau\rvert \leq \tau_0$, and $\lvert t\rvert \leq t_0$,
the operator $\Aopera_{s,t}$ is quasi-compact on the space $\BanachHolo$.
Within this region, there exist a simple, dominant eigenvalue $\lambda(s,t)$,
a rank-one operator $\myprojector_{s, t}$, and an operator $\mynihl_{s,t}$ such that
\[
\Aopera_{s,t} = \lambda(s,t)\,\myprojector_{s, t} + \mynihl_{s, t}.
\]
Moreover, the orthogonality relations
\[
\mynihl_{s, t}\myprojector_{s, t} = \myprojector_{s, t}\mynihl_{s, t} = 0
\]
hold. The operator $\myprojector_{s, t}$ is the projection onto the one-dimensional
eigenspace associated with $\lambda(s,t)$, and the residual spectral radius satisfies
\[
\varrho_{s,t}=rad(\mynihl_{s,t}) < \lvert \lambda(s,t)\rvert -\epsilon,
\]
for some $\epsilon=\epsilon(\delta, t_0, \tau_0)>0$ depending only on $\delta, \tau_0, t_0.$ In particular, we have
\begin{align*}
    \Aopera_{s,t}^n = \lambda(s,t)^n\myprojector_{s, t}^n+\mynihl^n_{s, t}.
\end{align*}
\end{prop}

\begin{proof}
By \cite[Theorem 4.8]{Kimetall}, the operator $\Aopera_{1,0}=\Myoperator^2$ is quasi compact. Then, the result above follows using \cite[Theorem 2.6]{Kloeckner-effective}.
\end{proof}
We also require some further properties of the eigenfunctions and eigenmeasures associated with $\Aopera_{s,t}$. We have the following analogue of \cite[Theorem 4.8]{Kimetall}. 
\begin{lem}\label{lem:spectralp2}
Let $(\delta, \tau_0, t_0)$ be as in \cref{prop:quasi-comp-swfamily}. Then, for the operator $\Aopera_{\sigma, t}$:
\begin{itemize}
    \item The eigenfunction $\psi_{\sigma, t}=(\psi_{\sigma, t,2}, \psi_{\sigma, t, 1}, \psi_{\sigma, t, 0})$ is positive, in the sense that $\psi_{\sigma, t, j}>0$ for all $j\in \{0,1,2\}$.
    \item The associated dual operator $\Aopera^\star_{\sigma,t}$ has a unique eigenmeasure $\nu_{\sigma, t}=(\nu_{\sigma, t,2}, \nu_{\sigma, t,1}, \nu_{\sigma, t,0})$ satisfying $\Aopera_{\sigma,t}^\star \nu_{\sigma, t}=\lambda(\sigma, t)\, \nu_{\sigma, t}$. In particular,  $\nu_{1,0,2}$ agrees with the usual two-dimensional Lebesgue measure.
\end{itemize}
\end{lem}
\begin{proof}
    This follows along the same lines of \cite[ Theorem 4.8]{Kimetall}. 
\end{proof}

From now on, in virtue of the above propositions, for $(s, t)$ in a neighborhood of $(1,0)$, we will call the function $(s,t)\to \lambda(s,t)$ the \textit{leading eigenvalue}. With this at hand we can state the main result of this section

\begin{prop}[The meromorphic continuation of $H(2s, t)$]\label{prop:genseriesameromorphiccontinuation}
There are parameters $\delta, \tau_0, t_0>0$ such that for all $\vert t\vert \leq t_0$, the function $s\mapsto H(s,t)$ has a meromorphic continuation to the right half plane
\begin{align}
    S_\delta=\{ s\in \CC,\; s=\sigma+i\tau,\; \sigma\geq 1-\delta\}.
\end{align}
Possible poles are located in the rectangular region delimited by $\vert \tau\vert \leq \tau_0$ and satisfy the condition $\lambda(s,t)=1$. Moreover, the function 
\begin{align}\label{eq:gen_series_mero_cont}
    s\mapsto H(2s,t)-\frac{\lambda(s,t)}{1-\lambda(s,t)}\left(\myid+ \Myoperator_{s, t}^{(1)}\right) \myprojector_{s, t}[1](0)
\end{align}
is holomorphic in $\Re(s)>1-\delta$, $\vert \tau\vert\leq \tau_0$. Additionally, the function $s\mapsto H(2s, t)$ is bounded by
\begin{align*}
  H(2s, t)\ll \vert \tau\vert^{C\vert \sigma-1\vert}
\end{align*}
in the strip $\Re(s)>1-\delta, \vert \tau\vert>\tau_0$, the  constant $C$ depending only over $(\delta, \tau_0, t_0).$
\end{prop}
\cref{prop:quasi-comp-swfamily}  almost suffices to prove the first part of the statement. The extra ingredient required is some control for the norm $\Vert \Aopera_{s,t}\Vert_a$ for $\tau \in [\tau_0, \tau_1]$ and all $\vert t\vert \leq t_0$ and \cref{lem:3-step-lemma}.
\subsection{Intermediate values of \texorpdfstring{$\tau$}{tau}}
We first need a short lemma.
\begin{lem}\label{lem:twist-on-the-line}
        Given any $\tau\not=0$ and $A$ being either $\Aopera_{1+i\tau, 0}$ or $\Myoperator_{1+i\tau, 0}$ we have $\Vert A\Vert_a<1$.
    
\end{lem}
\begin{proof}
    The proof follows in the same lines as \cite[p. 476]{Vall-euclidean-analysis}.
\end{proof}

We can now prove the following lemma.
\begin{lem}[Intermediate values of $\tau$]
Given any two $\tau_0, \tau_1$ with $\tau_0\leq \tau_1$, there is a $\delta>0$ and $t_0>0$ such that for any triple $(t, \sigma, \tau)$ with $\vert t\vert\leq t_0$, $\sigma\geq 1-\delta$ and $\tau\in [\tau_0, \tau_1]$ we have
\begin{align*}
\Vert \Aopera_{s,t} \Vert_a\leq 1-\delta.
\end{align*}
\end{lem}

\begin{proof}
    The proof follows the same lines as that of \cite[Lemma 6.3]{BettinDrappeau-main}. Consider the map $\tau \mapsto \Vert \Aopera_{1+i\tau, 0}\Vert_a$. By \cref{lem:3-step-lemma}, this map is continuous. By \cref{lem:twist-on-the-line} and continuity, we deduce the existence of an $\epsilon=\epsilon(\tau_0, \tau_1)>0$ such that $\Vert \Aopera_{1+i\tau, 0}\Vert_a<1-\epsilon$. Using  \cref{lem:3-step-lemma} again, we conclude.    
\end{proof}

\subsection{Bounds for the large values of \texorpdfstring{$\tau$}{tau}}
Let us normalize the operator $\Aopera$ as
\[
(\tilde{\Aopera}_{s,t}[f])_P = \frac{(\Aopera_{s,t}[\psi_{\sigma, t} f])_P}
                                   {\lambda(\sigma, t)(\psi_{\sigma, t})_P}.
\]
Let $\tilde{\varrho}_{\sigma, t}$ denote the spectral radius of the operator
$\frac{1}{\lambda(\sigma, t)} \Aopera_{\sigma,t} - \myprojector_{\sigma, t}$. The dual operator $\tilde{\Aopera}_{\sigma, t}^\star$  fixes the probability measure 
\begin{align}\label{def:mu-prob-meas}
\mu_{\sigma,t}:=\psi_{\sigma, t} \nu_{\sigma,t}.
\end{align}
The goal of this section is to establish the following bound.

\begin{prop}[Dolgopyat-type estimate]\label{prop:LV-bound}
Let $\delta, t_0 > 0$ be sufficiently small. Then there exist constants $C,  \tau_1 > 0$ and a sufficiently small positive constant $\beta$ such that for any $n = \lfloor C \log |\tau| \rfloor$,
\[
\bigl\Vert (I - \Aopera_{s,t})^{-1} \bigr\Vert_{c(\tau)} \ll |\tau|^{\beta},
\]
for any $s = \sigma + i\tau$ with $|\sigma-1| \leq \delta$, $|\tau| \geq \tau_1$ and $|t| \leq t_0$. 
\end{prop}

The key step in proving \cref{prop:LV-bound} is to establish the following inequality.

\begin{prop}[$L^2$ bound]\label{prop:L2}
Let $\delta, t_0$ be as in \cref{prop:LV-bound}, let $s=\sigma+i\tau$ with $\lvert \tau\rvert \geq \tau_1$ for some sufficiently large $\tau_1$. Specify the measure
\begin{align}\label{eq:defofmu}
    d\mu:=\psi_{1,0,2} \,dx\,dy.
\end{align}
Then, there are constants $K,\beta>0$, such that the bound
\begin{align*}
    \int_I \lvert \widetilde{\Aopera}_{s,t}^n f\rvert^2 \,d\mu \ll \frac{\lVert f \rVert_{c(\tau)}}{\tau^\beta}
\end{align*}
holds, provided $n=\lfloor K\log \lvert \tau\rvert \rfloor$.
\end{prop}

The above bound is an analogue of \cite[Theorem 6.1]{Kimetall}. Its proof follows an established strategy initiated by Dolgopyat \cite{Dolgopyat}, which has been adapted to number-theoretic problems by several authors (see e.g.\ \cite{BettinDrappeau-main}, \cite{Kimetall}, \cite{Vall-euclidean-analysis}). Many other steps follow \textit{mutatis mutandis} from the corresponding analysis in \cite[Chapters~5--6]{Kimetall}. However, for the benefit of the reader, we have decided to include them here, restricting ourselves to the meaningful modifications. Most of the lemmas playing a key role in these estimates, such as the contraction property, the UNI property, and the Van der Corput lemma, are unaffected by the twist in our operator.
\begin{lem}[Lasota--Yorke type inequality]\label{lem:lasotayorke}
Let $(s,t)$ satisfy $\lvert s-1\rvert \leq \delta$ and $\lvert t\rvert\leq t_0$ for some pair of sufficiently small positive numbers $(\delta, t_0)$. Then, for $f\in \BanachHolo$ and some constants $C_{\delta, t_0}, C'_{\delta, t_0}>0$, we have:
\begin{align}
&\|\Aopera_{\sigma,t}^n f\|_{c(1)} \leq C_{\delta, t_0} \bigl(\lvert s\rvert\,\|f\|_a + \rho^n \|f\|_{c(1)}\bigr),\label{eq:LS-nonnormalized}\\
&\|\tilde{\Aopera}_{s, t}^n f\|_{c(1)} \leq C'_{\delta, t_0} \bigl(\lvert s\rvert\,\|f\|_a + \rho^n \|f\|_{c(1)}\bigr),\label{eq:LS-normalized}\\
&\|\tilde{\Aopera}_{1,0}^n f\|_a = \int_{I_D} f\,d\mu_{1,0} + \LandauO\bigl(\varrho_{1,0}^n \|f\|_{c(1)}\bigr).\label{eq:equilibrium-rep}
\end{align}
where $\varrho_{1,0}$ is the spectral radius of $\Aopera_{1,0}-\myprojector_{1,0}$.
\end{lem}

\begin{proof}[Sketch of proof of \cref{lem:lasotayorke}]
We may focus on the $\Vert \cdot\Vert_b$ factor, as the statement is trivial for the $\Vert \cdot\Vert_a$ norm. For the non-normalized operator, we see that for $P\in \Mpartition[1]\cup \Mpartition[2]$ we have
\begin{align}
    (\Aopera_{\sigma, t}^n f)_P(z)=\sum_Q \sum_{ \tick{\alpha} \in H^{2n}(P,Q)} e^{itc(\tick{\alpha})} \lvert J_{\tick{\alpha}}(z)\rvert^\sigma f_Q \circ \tick{\alpha}(z),
\end{align}
where
\begin{align}\label{eq:c-alpha-starstar}
    c(\tick{\alpha})
    := \bigl[ \Psi \circ (\tick{\alpha}_1 \circ \cdots \circ \tick{\alpha}_n) 
                     + \Psi \circ (\tick{\alpha}_2 \circ \cdots \circ \tick{\alpha}_n) 
                     + \cdots 
                     + \Psi \circ \tick{\alpha}_n \bigr],
\end{align}
for $\tick{\alpha}_i\in B^2(Q_i, Q_{i+1})$ and $\tick{\alpha}=\tick{\alpha}_1 \circ \cdots \circ \tick{\alpha}_n\in B^{2n}(P,Q)$, see also \cref{rem:quicknote}.  At this point, we recall that, by our choice of the period function $\Psi$, the value of $\Psi\circ \tick{\alpha}(z)$ does not depend on $z$ for $z\in P$. Estimating $\lvert\partial_v (\Aopera_{\sigma,w}^n f)_P\rvert$ reduces to controlling the following two terms:
\begin{align*}
&(\Sigma_1)=\left\lvert \sum_{Q\in\mathcal{P}}\sum_{\tick{\alpha}}
   \lvert\sigma\rvert\,\lvert J_{\tick{\alpha}}\rvert^{\sigma-1}\lvert\partial_v J_{\tick{\alpha}} (z)\rvert
   \cdot (f)_Q\circ\tick{\alpha} \right\rvert \\
&(\Sigma_2)=\left\lvert \sum_{Q\in\mathcal{P}}\sum_{\tick{\alpha}} 
   \lvert J_{\tick{\alpha}}\rvert^\sigma (\partial_v (f_Q\circ\tick{\alpha})(z) \right\rvert,
\end{align*}
where the inner sum is taken over $\tick{\alpha}\in\mathcal{H}^{2n}(P,Q)$. The sum in $(\Sigma_1)$ is bounded by $M\lvert \sigma\rvert \Vert f\Vert_a$, where $M$ is the distortion constant in \cite[Proposition 2.4]{Kimetall}, while the sum in $(\Sigma_2)$ is bounded by $\rho^n \Vert f\Vert_b$, for $\rho$ the contraction ratio from \eqref{eq:contractionratio}. Gathering everthing together (and taking supremum over all tangent unit vectors and over all parts) gives~\eqref{eq:LS-nonnormalized}. For the normalized operator $\tilde{\Aopera_{s,t}}$, set
\[
A(\delta, t_0)=\sup_{\sigma, t}\bigl( \Vert \psi_{\sigma, t} \Vert_b \Vert \psi_{\sigma, t}^{-1}\Vert_a\bigr).
\]
Then the term $|\partial_v (\Aopera_{s,t}^n f)_P|$ consists of three terms:
\begin{align*}
&\lambda(\sigma,t)^{-n}\cdot\frac{\partial_v(\psi_{\sigma,t})_P}{(\psi_{\sigma,t})_P^2}
   \sum_{Q\in\mathcal{P}}\sum_{\tick{\alpha}} e^{itc(\tick{\alpha})}
   |J_{\tick{\alpha}}|^s \cdot (\psi_{\sigma,t}\cdot f)_Q\circ\tick{\alpha} \qquad &\text{(I)},\\
&\lambda(\sigma,t)^{-n}(\psi_{\sigma,t})_P
   \sum_{Q\in\mathcal{P}}\sum_{\tick{\alpha}}  e^{itc(\tick{\alpha})}
   |s|\,|J_{\tick{\alpha}}|^{s-1}|\partial_v \tilde{J}_{\tick{\alpha}}|
   \cdot (\psi_{\sigma,t}\cdot f)_Q\circ\tick{\alpha} \qquad &\text{(II)},\\
&\lambda(\sigma,t)^{-n}(\psi_{\sigma,t})_P\sum_{Q\in\mathcal{P}}\sum_{\tick{\alpha}}  e^{itc(\tick{\alpha})}
   |J_{\tick{\alpha}}|^s (f\cdot\partial_v\psi_{\sigma,t} + \psi_{\sigma,t}\cdot\partial_v f)_Q\circ\tick{\alpha} \qquad &\text{(III)}.
\end{align*}
One obtains
\begin{align*}
   \vert \mathrm{I}\vert  \leq A(\delta, t_0) \Vert  \tilde{\Aopera}^n_{\sigma, t} \vert f\vert\Vert_a\leq A(\delta, t_0) \Vert \vert f\vert \Vert_a.
\end{align*}
For term~(II) one obtains the bound
\begin{align*}
    \mathrm{(II)}\leq M \vert s\vert  \Vert \vert f\vert\Vert_a,
\end{align*}
and for the last term one has
\begin{align*}
\ll \rho^n\|f\|_0 + \rho^n\|f\|_1.
\end{align*}
For \eqref{eq:equilibrium-rep}, assume that the eigenfunction and measure are normalised, i.e., $\int_I \psi_{\sigma,t}\,\nu_{\sigma,t}=1$. The spectral theorem gives
\begin{align*}
    \Aopera_{\sigma, t}^n =\lambda(\sigma, t)^n \myprojector_{\sigma, t}+\mynihl^n_{\sigma, t},
\end{align*}
and \eqref{eq:equilibrium-rep} then follows along the same lines as in \cite[Lemma 5.1]{Kimetall}.
\end{proof}

\begin{proof}[Sketch of proof of \cref{prop:L2}]
We have 
\[  \int_I |\widetilde{\Aopera}_{s,t}^n f|^2 d\mu= \lambda(\sigma,t)^{-2n}\sum_{P \in \Mpartition[2]} 
\int_P  \left(\psi_{\sigma,t}^{-2}\right)_P |{\Aopera}_{s,t}^n (\psi_{\sigma,t}  f))_P|^2 dxdy  \]
 by definition. Expanding, we have
\begin{align}\label{L2:expansion}
    I_P := \sum_{Q \in \Mpartition[2]} \sum_{\substack{\tick{\alpha}, \tick{\beta}\in B^{2n}(P,Q)}}  
        \int _P e^{it(c(\tick{\alpha})-c(\tick{\beta}))} e^{i\tau \phi_{\tick{\alpha}, \tick{\beta}}} R^\sigma_{\tick{\alpha},\tick{\beta}} \, dx\,dy.
\end{align}
where
\begin{align*}
    R_{\tick{\alpha},\tick{\beta}}^\sigma &:= \left(\psi_{\sigma,t}^{-2}\right)_P |J_{\tick{\alpha}}|^\sigma |J_{\tick{\beta}}|^\sigma  \cdot (\psi_{\sigma,t} \cdot f)_Q \circ \langle \tick{\alpha} \rangle \cdot (\psi_{\sigma,t} \cdot \bar{f})_Q \circ \langle \tick{\beta} \rangle, \\
    \phi_{\tick{\alpha}, \tick{\beta}}&:=\log|J_{\tick{\alpha}}|-\log|J_{\tick{\beta}}|.
\end{align*}
and $c(\tick{\alpha})$ is as in \eqref{eq:c-alpha-starstar}.
From here, the estimation runs almost verbatim with respect to \cite[Section 6]{Kimetall}. For $\tick{\alpha},  \tick{\beta} \in B^n(P,Q)$, and any $n\in \NN_{>0}$, define the distance 
\begin{align*}
\Delta(\tick{\alpha},\tick{\beta}) := \inf_{(x,y) \in P} 
\Vert \left(\partial_z\phi_{\tick{\alpha},\tick{\beta}}(x,y),\partial_{\bar z} \phi_{\tick{\alpha},\tick{\beta}}(x,y)\right)\Vert,
\end{align*}
where $\partial_z$ and $\partial_{\bar z}$ respectively denote the derivative in $z=x+iy$ and $\bar z=x-iy$. Decompose $I_P$ as $I_P = I_{P,1}+I_{P,2}$ where
\begin{align*}
    I_{P,1} &:=\sum_{Q \in \Mpartition[2]} \sum_{\substack{\Delta(\tick{\alpha},\tick{\beta}) \leq \kappa\\ \tick{\alpha}, \tick{\beta} \in B^{2n}(P,Q)}} e^{it(c(\tick{\alpha})-c(\tick{\beta}))} \int_P 
     e^{i\tau \phi_{\tick{\alpha}, \tick{\beta}}} R_{\tick{\alpha},\tick{\beta}}^\sigma dxdy
\end{align*}
and
\begin{align*}
    I_{P,2} &:=\sum_{Q \in \Mpartition[2]} \sum_{\substack{\Delta(\tick{\alpha},\tick{\beta}) > \kappa\\ \tick{\alpha}, \tick{\beta} \in B^{2n}(P,Q)}} e^{it(c(\tick{\alpha})-c(\tick{\beta}))} \int_P
     e^{i\tau \phi_{\tick{\alpha}, \tick{\beta}}} R_{\tick{\alpha},\tick{\beta}}^\sigma dxdy,
\end{align*}
for some parameter $\kappa>0$ to be defined later. Notice that by the choice of $\Psi$ in \eqref{eq:eqperiod}, $c(\tick{\alpha})$ is real, so that $\vert e^{it c(\tick{\alpha})}\vert=1$. Set now $\kappa=\rho^{\gamma n}$, where $\rho$ is the contraction ratio in  \eqref{eq:contractionratio}, $\gamma\in[0,1]$ is some small parameter that will select later. Following the same steps in \cite[Subsection 6.2]{Kimetall}, we obtain
\begin{align*}
    \vert I_{P,1}\vert &\ll \Vert f\Vert_a^2 \sum_{Q \in \Mpartition[2]} \sum_{\substack{\Delta(\tick{\alpha},\tick{\beta}) \leq \kappa\\ \tick{\alpha}, \tick{\beta} \in B^{2n}(P,Q)}}\int_P 
     J_{\tick{\alpha}}^\sigma J_{\tick{\beta}}^\sigma dxdy\\
     &\ll\Vert f\Vert_a^2  \rho^{\gamma n}.
\end{align*}
For  $I_{P,2}$, following the discussion in \cite{Kimetall} and applying a suitable Van der Corput type estimate, one obtains
\begin{align*}
    I_{P,2}\ll \Vert f\Vert_{c(\tau)}^2 \frac{(1+\tau\rho^n)}{\tau}\left( \frac{C_1}{\kappa}+\frac{C_2}{\kappa^2}\right)=\Vert f\Vert_{c(\tau)}^2 \frac{(1+\tau\rho^n)}{\tau}\left( \frac{C_1}{\rho^{2\gamma n}}+\frac{C_2}{\rho^{4\gamma n}}\right),
\end{align*}
where $C_1=\text{len}(\partial P)+\vol(P)$ and $C_2$ is proportional to $\vol(P)$,  $2n$ is the depth of the branches appearing in $I_{P}$. To conclude, take $n\approx \lfloor \log \vert \tau \vert \rfloor$ and $\gamma<\frac{1}{4}\log \frac{1}{\rho}$.
\end{proof}

We can now prove the Dolgopyat-type estimate.
\begin{proof}[Proof of \cref{prop:LV-bound}]
By Cauchy-Schwarz
\begin{align*}
\vert \widetilde{\Aopera}^n_{s,t} f\vert^2&= 
\frac{1}{\lambda(\sigma, t)^{2n} (\psi_{\sigma, t})^2_P} 
\left(\sum_{Q\in \Mpartition} \sum_{ \fatalpha\in B^{2n}(P,Q)} J_{\fatalpha}^\sigma \cdot (\psi_{\sigma, t} f) \circ \alphares \right)^2 \\
&\leq 
\frac{1}{\lambda(\sigma, t)^{2n} (\psi_{\sigma, t})^2_P} 
\left(\sum_{Q\in \Mpartition} \sum_{\fatalpha\in B^{2n}(P,Q)} J_{\fatalpha}^{2\sigma-1}\right) 
\left(\sum_{Q\in \Mpartition} \sum_{\fatalpha\in B^{2n}(P,Q)} J_{\fatalpha} \cdot \bigl(\vert \psi_{\sigma, t} f\circ \alphares\bigr\vert^2\right),
\end{align*}
where we used again the fact that $c(\fatalpha)$ has modulus $1$.  We recognize the second sum as $\Aopera_{1,0}(\vert \psi_{\sigma, t} f\vert^2)_P$. By \eqref{eq:equilibrium-rep} we see
\begin{align}
    \frac{1}{\Vert \psi_{\sigma,t}^2\Vert_a}\Vert \widetilde{\Aopera}^n_{1,0} \vert \psi_{\sigma, t} f\vert \Vert_a\leq  \int_I \vert  f\vert^2 dxdy + \LandauO\left(\varrho_{1,0}^n \Vert f\Vert_{c(1)}\right).
\end{align}\label{eq:LV-bound-projection-reduction}
For the first sum, we notice that it can be rewritten as
\begin{align*}
    \left(\frac{\lambda(2\sigma-1, 0)}{\lambda(\sigma, t)^2}\right)^n \frac{1}{\psi_{\sigma,t}^2} (\widetilde{\Aopera}^n_{2\sigma-1, 0}( \psi_{2\sigma-2, 0}^{-1}))_P.
\end{align*}
By taking the supremum norm, and using the fact the operator is normalized, we have
\[
\left\vert\frac{\lambda(2\sigma-1, 0)}{\lambda(\sigma, t)^2}\right\vert
\left\Vert\frac{1}{\psi_{\sigma,t}^2} \cdot \widetilde{\Aopera}_{2\sigma-1, 0}\!\bigl(\psi_{2\sigma-2, 0}^{-1}\bigr)\right\Vert_a
\leq L_{\sigma, t}.
\]
where $L_{\sigma,t}$  satisfises $L_{1,0}=1$ (by the normalization) and is continuous in the parameters $\sigma, t$.  Thus,
\[
\left\vert\frac{\lambda(2\sigma-1, 0)}{\lambda(\sigma, t)^2}\right\vert^n
\Bigl\Vert\frac{1}{\psi_{\sigma,t}^2} \cdot \widetilde{\Aopera}^n_{2\sigma-1, 0}
\bigl(\psi_{2\sigma-2, 0}^{-1}\bigr)\Bigr\Vert_a
\le L^n_{\sigma, t}.
\]
Therefore, we have
\begin{align*}
    \Vert \widetilde{\Aopera}^n_{s,t} f \Vert_a^2 &\leq L_{\sigma,t}^{n} \Vert \widetilde{\Myoperator}^{2n} |f|^2 \Vert_a.
\end{align*}
Now, write $n = n_1 + n_2$, and plug $\widetilde{\Aopera}_{s,t}^{n_2} f$ instead of $f$ in the above, use \cref{prop:L2}, equation \eqref{eq:equilibrium-rep} and then equation \eqref{eq:LS-normalized}. This gives
\begin{align*}
    \Vert \widetilde{\Aopera}^{n_1 + n_2}_{s,t} f \Vert_a^2 &\leq C_{\delta, t}
    L_{\sigma,t}^{n_1} 
    \left( \int_I |\widetilde{\Aopera}^{n_2}_{s,t} f|^2 dxdy 
    + \varrho_{1,0}^{n_1} |\tau| \left( \Vert f \Vert_a 
    + \rho^{n_2} \Vert f \Vert_{c(\tau)} \right) \right)\\
    &\leq C_{\delta, t}
    L_{\sigma,t}^{n_1} 
    \left( \int_I |\widetilde{\Aopera}^{n_2}_{s,t} f|^2 dxdy 
    + \varrho_{1,0}^{n_1} |\tau| \Vert f \Vert_{c(\tau)} \right).
\end{align*}
Using now \cref{prop:L2} with $n_2=\lfloor K\log \vert \tau \vert \rfloor $ for some $\beta, K>0$ small enough, we have
\begin{align*}
    \Vert \widetilde{\Aopera}^{n_1 + n_2}_{s,t} f \Vert_a^2 \ll C_{\delta, t}  L_{\sigma,t}^{n_1} 
   \Vert  f \Vert_{c(\tau)} \left( \frac{1}{\vert \tau\vert^{\beta}}
    + \varrho_{1,0}^{n_1} |\tau|  \right).
\end{align*}
At this point, notice that $\tau = e^{n_2/K + \LandauO(1)}$ by our selection. 
Let us define 
\[L(\delta, t_0) := \sup_{\substack{|\sigma-1| \leq \delta, |t| \leq t_0}} L_{\sigma,t}.\] 
We want to ensure we can select parameters such that $L^{n_1} \ll |\tau|^{-\frac{\beta}{2}}$ 
and $|\tau| \ll \varrho_{1,0}^{-n_1/2}$. To this end, let $n_2 = C n_1$ for some parameter $C$. 
Then the above is equivalent to the following two inequalities:
\[
\begin{cases}
    \displaystyle\frac{1}{2}\log L(\delta, t_0) - \frac{C\beta}{2K} \leq 0,\\[8pt]
    \displaystyle\frac{1}{2} \log \varrho_{1,0} + \frac{C}{K} \leq 0.
\end{cases}
\]
By quasi-compactness, the residual spectrum  $\rho$ satisfises $\varrho<1$. Without loss of generality, $L(\delta, t_0)>1$. Then set $C=\frac{2K \log L(\delta,t_0)}{\beta}$. The inequalities are readily verified provided we take $\delta, t_0$ small enough such that $\log L(\delta, t_0)$ is small enough. This proves that if $\delta, t_0$ are sufficiently small and $n=\lfloor K'\log \tau\rfloor$ we have
\begin{align*}
    \Vert \widetilde{\Aopera}^{n_1 + n_2}_{s,t} f \Vert_a^2 \ll 
   \tilde{\rho}^n  \Vert f \Vert_{c(\tau)},
\end{align*}
for some $\tilde{\rho}\in (0,1)$. From here, one can conclude following the same lines as \cite[Lemma 6.4]{BettinDrappeau-main}.
\end{proof}

\subsection{Proof of Proposition~\ref{prop:genseriesameromorphiccontinuation}}
\begin{proof}
    We have the representation
    \begin{align*}
        H(2s,t)= \left(\myid+\Myoperator_{s, t}^1\right) \sum_{n\geq 0} \Aopera^n_{s, t} [1] (0)
    \end{align*}
    which is initially valid for $\Re(s)>1$. Now, by \cref{prop:quasi-comp-swfamily}, we know that there are parameters $\delta, \tau_0, t_0>0$ and operators $\myprojector_{s, t}$ and $\mynihl_{s, t}$ for which, for any $f\in \BanachHolo$
    \begin{align}\label{eq:somequationtoreference}
        \Aopera_{s,t}^n[f](z)=\lambda(s,t)^n \myprojector_{s, t}[f](z)+\mynihl_{s, t}^n [f](z).
    \end{align}
    Now, we can choose the $\delta$ such that the spectral radius of $\mynihl_{s, t}$ is less than $1-\delta$. We take now the sum over $n$ in \eqref{eq:somequationtoreference} to obtain
    \begin{align*}
        H(2s, t)= \left(\myid+\Myoperator_{s, t}^{(1)}\right)
        \left(\myid + \frac{\lambda(s,t)}{1-\lambda(s,t)} \myprojector_{s, t} +\sum_{n\geq 1} \mynihl_{s, t}^n \right) [1] (0).
    \end{align*}
    Plainly, $\left(\myid+\Myoperator_{s, t}^{1}\right) 
        \left(\myid +\sum_{n\geq 1} \mynihl_{s, t}^n \right) [1] (0)$ 
    is analytic in $\Re(s)>1-\delta$ and $\vert t\vert\leq t_0$. 
    This proves \eqref{eq:gen_series_mero_cont}. 
    The proof of the upper bound for $H(2s,t)$ follows by the property of norms and  \cref{prop:LV-bound}.
\end{proof}

\section{Isolating singularities of \texorpdfstring{$H(2s,t)$}{H(2s,t)}}\label{sec:lambda-asypt}
In the following section we gain further information on the behavior of the leading eigenvalue $\lambda(s,t)$ around the point $(1,0)$.  This section is analogous to \cite[Section 7]{BettinDrappeau-main} and pplays a key role later in the proof of the asymptotic expansion of $\chi_S(t)$, see \cref{prop:asymptotic:expansion}. 
\begin{lem}[Asymptotic expansion for $\partial_s \lambda(s,t)$]\label{lem:partial_lambda_asymptotic}
For any sufficiently small $\delta>0$, there is some positive $t_0>0$, such that whenever $|s - 1| \leq \delta$ and $|t| \leq t_0$, the following asymptotic holds:
\begin{align*}
    \partial_s \lambda(s,t)= \int_{I_D} \log  T  d\mu +\LandauO(\delta),
\end{align*}
where $T$ is the factor in \eqref{eq:gst} and $d\mu=\psi_{(1,0,2)} \,dx\,dy$ is the measure in \eqref{eq:defofmu}.
\end{lem}

\begin{proof}
Let $\psi_{s, t}\in \BanachHolo$ be the eigenfunction of $\Aopera_{s,t}$ associated with the eigenvalue $\lambda(s,t)$. Therefore
\begin{align}\label{eq:eigenwertequation}
\Aopera_{s,t} \psi_{s, t}=\lambda(s, t) \psi_{s, t}.
\end{align}
Notice that for any function $f\in \BanachHolo$ we have
\[
\int_I \Aopera_{s,t}(f)(x,y)\,dx\,dy = \int_I \Myoperator^2[g_{s, t} f] (x,y) \,dx\,dy = \int_I (g_{s, t}f)(x,y) \,dx\,dy.
\]
Thus for $f=\psi_{s, t}$ in the above equation, we have
\[
\lambda(s,t) \int_I \psi_{s, t}(x,y) \,dx\,dy =\int_I (g_{s, t}\psi_{s, t})(x,y) \,dx\,dy.
\]
By differentiating the above equation with respect to $s$ we obtain
\[
\partial_{s}\lambda(s,t) \int_I \psi_{s, t} (x,y) \,dx\,dy = 
\int_I \Bigl( (\log T(z))g_{s, t} \psi_{s, t} 
+ (g_{s, t} - \lambda(s,t))\,\partial_{s}\psi_{s, t} \Bigr) (x,y)\,dx\,dy.
\]
Setting $(s,t)=(1,0)$ we get $g_{1,0}=1$ and $\lambda(1,0)=1$; hence
\begin{align*}
    \partial_{s}\lambda(s,t)\vert_{(s=1, t=0)} \int_I \psi_{1,0} \,dx\,dy= \int_{I} \log T \psi_{1,0} \,dx\,dy =\int_I \log T \psi_{(1, 0, 2)} \,dx\,dy.
\end{align*}
Since $\myprojector_{s, t}\psi_{s, t}=\psi_{s, t}$ using \cite[Theorem 2.6]{Kloeckner-effective}, we get  $\myprojector_{s, t}=\myprojector_{1,0}+\LandauO_{s, t}\bigl(\Vert \Aopera_{s,t}-\Aopera_{1,0} \Vert_a\bigr)$. Therefore, by using \cref{lem:3-step-lemma}, we have
\begin{align*}
    \Vert \psi_{s, t}-\psi_{1,0}\Vert_a\ll \vert s-1\vert + \vert t\vert \ll \delta,
\end{align*}
provided $t$ is selected accordingly. 
\end{proof}

\begin{lem}
\label{lem:delta-1}
Let $\tau>0$ be a parameter.  There is a sufficiently small $t_0>0$ such that the (unique) function $s_0$ defined implicitly by $s_0(0)=1$ and $\lambda(s_0(t), t)=1$, satisfies
\begin{align*}
    \vert s_0(t)-1\vert \leq \tau,
\end{align*}
provided $\vert t\vert \leq t_0.$
\end{lem}

\begin{proof}
    Follows \textit{mutatis mutandis} along the lines of \cite[Lemma 4.10]{Kimetall} or \cite[Lemma 7.2]{BettinDrappeau-main}. The hypotheses are satisfied by Lemma \eqref{lem:partial_lambda_asymptotic}.
\end{proof}

\section{The asymptotic expansion of the characteristic function}\label{sec:expansion-char-function}
Now, we extract  an asymptotic expansion for the characteristic function $\chi_S$ for small $t$. 

\begin{prop}[Asymptotic expansion of the characteristic function]\label{prop:asymptotic:expansion}
    There are parameters $\delta, t_0>0$ such that for $|t| \leq t_0$ the following holds
\begin{align*}
    \chi_{X}(t) =X^{2s_0(t)-2}  \bigl(1+ \LandauO\bigl( \frac{1}{X^{\delta}} + |t\vert \bigr)\bigr),
\end{align*}
where $s_0$ is the function in \cref{lem:delta-1}.
\end{prop}

\begin{proof}
We follow \cite{BettinDrappeau-main}. We define the smoothed quantity
\[
A(t) := \sum_{n \leq X} a_{t,n} \, f\!\left(\frac{n}{X}\right),
\]
where \(f: \mathbb{R}_{\geq 0} \to [0,1]\) is a smooth non-negative function satisfying:
\begin{enumerate}
    \item $f \equiv 1$ on $[0,1]$,
    \item $f(x) \leq \mathbf{1}_{[0, 1+1/W^4]}(x)$, for $W>0$ to be specified later,
    \item $\left\|f^{(k)}\right\|_{\infty} \ll_{k} W^{4k}$ for all $k \geq 0$.
\end{enumerate}
Let \(\hat{f}(s)\) denote the Mellin transform of \(f\). By Mellin inversion,
\begin{align*}
    f\!\left(\frac{n}{X}\right) 
    &= \frac{1}{\pi i} \int_{\sigma_0 - i\infty}^{\sigma_0 + i\infty} 
       \left(\frac{X}{n}\right)^{2s} \hat{f}(2s) \, ds,
\end{align*}
for some \(\sigma_0> 2\). Whence
\begin{align*}
    A(t) = \frac{1}{\pi i} \int_{\sigma_0 - i\infty}^{\sigma_0 + i\infty} X^{2s} \, H(2s, t) \, \hat{f}(2s) \, ds.
\end{align*}

By \cref{prop:genseriesameromorphiccontinuation}, there exist constants $\delta, t_0 > 0$ such that for $|t| \leq t_0$, the function $s \mapsto H(2s,t)$ extends meromorphically to the half-plane $\Re(s) \geq 1 - \delta$. By choosing $\delta$ and $t_0$ sufficiently small, $H(2s,t)$ has only one simple pole, located at $s = s_0(t)$. Using the residue theorem and the decay estimates from \cref{prop:genseriesameromorphiccontinuation}, we can shift the contour of integration from $\Re(s) = \sigma_0$ to $\Re(s) = 1 - \delta$. By \cref{lem:delta-1},  and provided $\delta$ and $t_0$ are small enough, this gives
\begin{align}\label{eq:perron1}
    A(t) = 2 \, \underset{s = s_0(t)}{\mathrm{Res}}\; 
    \Bigl[ X^{2s} H(2s,t) \hat{f}(2s) \Bigr]
    \;+\; \frac{1}{\pi i} \int_{1-\delta - i\infty}^{1-\delta + i\infty} X^{2s} H(2s,t) \hat{f}(2s) \, ds.
\end{align}
 Using the smoothing properties of $f$ listed above and \cref{prop:genseriesameromorphiccontinuation}, we can estimate the integral on the right side above as $\ll_k X^{2(1-\delta)} W^{4k}$ for some integer $k$ large enough. The residue part is
\begin{align}\label{eq:residuaintermediate}
    -\frac{X^{2s_0(t)}\hat{f}(2s_0(t))}{\partial_s \lambda(s_0(t),t)}\left(\myid+\Myoperator_{s_0(t),t}^{(1)}\right) \myprojector_{s_0(t), t}[1](0).
\end{align}
At this point we use the following approximation, \cite[see Theorem 1.6]{Kloeckner-effective}:
\begin{align*}
    \myprojector_{s, t}=\myprojector_{1,0}+ \LandauO\left(\Vert \Aopera_{s,t}-\Aopera_{1,0}\Vert_a\right),
\end{align*}
together with \cref{lem:3-step-lemma}, to obtain
\begin{align}
    \Vert \Myoperator_{s, t}^{(1)}-\Myoperator\Vert_a \ll \vert t\vert.
\end{align}
Since
\begin{align*}
    \Myoperator_{1, 0} = \myprojector_{1,0}+ \mynihl_{1,0}\,, \quad 
    \myprojector_{1,0}\mynihl_{1,0}=\mynihl_{1,0}\myprojector_{1,0}=0,
\end{align*}
we get 
\begin{align*}
    \left(\mathbb{I}+\Myoperator_{s_0(t), t}^{(1)} \right) \myprojector_{s_0(t), t}[1](0)
    &= \myprojector_{1,0}[1](0)+\myprojector_{1,0}^2 [1](0) 
       + \LandauO\bigl( |t| \bigr)\\
    &= 2\myprojector_{1,0}[1](0) + \LandauO\bigl( |t|\bigr).
\end{align*}
Now, set $C=2\myprojector_{1,0}[1](0)$, and notice that this, being an integral of a positive function with respect to the measure $\mu$, is positive. Notice now that for $\delta, t_0$ sufficiently small, we have
\begin{align*}
    \hat{f}(2s_0)=\frac{1}{2s_0}+ \int_1^{1+\frac{1}{W^4}} f(x)x^{2s_0(t)-1} dx=\frac{1}{2s_0}\left(1+\LandauO\left(\frac{1}{W^4}\right)\right) =\frac{1}{2}\left(1+\LandauO(\frac{1}{W^4}+ \vert t\vert)\right).
\end{align*}
Notice now that $\sharp \FFFF(X)=C'X^2+\LandauO(X)$ and that by smoothing properties above, we have
\[
A(t)=\sum_{v\in \FFFF(X)} e^{it S(v)}+ \LandauO\left(X^{1+\epsilon} \left(1+\frac{X}{W^4}\right)\right).
\]
Selecting now $W=X^{\frac{2\delta}{4k+1}}$, using the expression for $\partial_s \lambda(s,t)$ from \cref{lem:partial_lambda_asymptotic} and normalizing, we get
\begin{align*}
    \chi_{X}(t) =X^{2s_0(t)-2}
    \biggl(1+ \LandauO\bigl( \frac{1}{X^{\delta}} + |t| + \vert 2s_0(t)-2\vert \bigr)\biggr).
\end{align*}
The result follows by taking $\delta, t_0$ small enough.
\end{proof}

We now aim to extract an asymptotic for $s_0(t)$ at small $t$. We prove the following result which gives a first step in this direction and reduces the problem to estimating an oscillatory integral.

\begin{lem}\label{lem:s0_asymptotic}
Let $|t| \leq t_0$ for some sufficiently small $t_0 > 0$. Then the following asymptotic holds:
\begin{align*}
    s_0(t) - 1 = \frac{1}{A} \sum_{P\in P[2]}\int_P (1-e^{it\Psi(z)}) \psi_{(1,0,2)}(z) \,dx\,dy + \LandauO\left( t^2 \right),
\end{align*}
where
\begin{align*}
    A = \sum_{P\in P[2]}\int_P \log T(z) \,\psi_{(1,0,2)}(z)  \,dx\,dy>0.
\end{align*}
\end{lem}

\begin{proof}
    Recall the decomposition 
    \begin{align*}
        \mathscr{C}(\Mpartition)=\mathscr{C}(\Mpartition[0])\oplus \mathscr{C}(\Mpartition[1])\oplus \mathscr{C}(\Mpartition[2])
    \end{align*}
    of the space of functions on which the family $\Aopera_{s,t}$ acts (see \cite[proof of Theorem 4.8]{Kimetall}). Let us denote by $\Myoperatordual$ the dual of the operator $\Myoperator$. It is possible to show (see \cite[Theorem 4.8]{Kimetall}) that to the eigenfunction $\tick{\phi}=(\phi_2, \phi_1, \phi_0)$ - which is exactly $\psi_{1,0}=(\psi_{1,0,0}, \psi_{1,0, 1}, \psi_{1,0,2})$ from \cref{prop:quasi-comp-swfamily} with $\lambda(1,0)=1$ - we have an associated eigenfunctional $\phi^\star: \mathscr{C}(\mathscr{P}) \to \CC$ (essentially induced by the eigenmeasure in \cref{lem:spectralp2}). Explicitly, this functional is given by
    \begin{align*}
        f=(f_0, f_1, f_2)\mapsto \phi^\star(f)=\sum_{P \in P[2]}\int_P f_2 \,dx\,dy,
    \end{align*}
    with $\Myoperatordual \phi^\star = \phi^\star$. Now, recall that 
    $g_{s, t}(z)$ is an analytic function in $s$. We therefore have the expansion
    \begin{align*}
        g_{s, t}(z)=e^{it\Psi(z)}+(s-1)\left(\partial_s g_{s, t}(z)\right)_{s=1}+ \LandauO\left(|s-1|^2\right).
    \end{align*}
    Now, by Theorem 2.6 in \cite{Kloeckner-effective}, we obtain    
    \begin{align*}
        \lambda(s, t)=\lambda(1, 0) + (\phi^\star \circ (\Aopera_{s,t}-\Myoperator^2)) (\tick{\phi}) + \LandauO\left(\Vert \Aopera_{s, t} -\Myoperator\Vert_{a}^2\right).
    \end{align*}
    Since $\lambda(1, 0)=1$, we see that $\phi^\star\circ \Myoperator^2 (\tick{\phi})= \phi^\star (\tick{\phi})=1$, by normalization. Thus we have
    \begin{align*}
        \lambda( s, t)= (\phi^\star \circ \Aopera_{s, t}) (\tick{\phi}) + \LandauO\left(\Vert \Aopera_{s, t}-\Myoperator^2\Vert_{a}^2\right).
    \end{align*}
   By \cref{prop:reduction} and basic computations, we obtain
    \begin{align*}
        \Aopera_{s,t}(\tick{\phi})&=\Myoperator^2( g_{s, t} \tick{\phi})\\
        &=\Myoperator^2[e^{it\Psi(z)}\tick{\phi}]+ (s-1)\Myoperator^2[\left(\partial_s g_{s, t}\right)\big|_{s=1} \tick{\phi}]+ \LandauO\left(|s-1|^2 \|\tick{\phi}\|_a\right),
    \end{align*}
    with
    \begin{align*}
        g_{s, t}(z)= e^{it\Psi(z)} \exp\left((s-1) \log T(z)\right).
    \end{align*}
    Thus
    \begin{align*}
        (s-1)\Myoperator^2[\partial_s g_{ s,t }\big|_{s=1} \tick{\phi}]&=(s-1)\Myoperator^2[ \log T(z) e^{it\Psi(z)} \tick{\phi}]\\
        &=(s-1) \Myoperator^2[\log T(z)\tick{\phi}]+ \LandauO\left( \|\tick{\phi}\|_a |s-1| |t|\right),
    \end{align*}
where in the last line we used that, by  \cref{lem:3-step-lemma},
\begin{align*}
    \Myoperator^2[e^{it\Psi(z)} \tick{\phi}]&=\Myoperator^2[\tick{\phi}]- \Myoperator^2[(1-e^{it\Psi(z)}) \tick{\phi}]=\tick{\phi} \left(1+ \LandauO\left(|t| \right)\right).
\end{align*}
This leads to 
\begin{align}
    \lambda(s, t) &= \phi^\star\!\left(\Myoperator^2[e^{it\Psi(z)} \tick{\phi}]\right) + (s-1) A \nonumber \\
    &\quad+ \LandauO\Bigl(|s-1|^2 \|\tick{\phi}\|_a + \|\tick{\phi}\|_a |s-1||t| 
           + \Vert \Aopera_{s,t}-\Myoperator^2\Vert_{a}^2\Bigr) \nonumber\\
    &= \sum_{P\in P[2]} \int_{P} e^{it\Psi(z)} \psi_{(1,0,2)}(z) \,dx\,dy + (s-1) A + \LandauO\bigl(|s-1|^2 + |t|^2\bigr)\nonumber.
\end{align}
Since $\lambda(s_0(t), t)=1$, we conclude by taking $s=s_0(t)$ in the above.
\end{proof}

\section{Asymptotic expansion of oscillatory integrals}\label{sec:integral-asypmt}
In view of \cref{lem:s0_asymptotic}, we need to understand 
\begin{align*}
\sum_{P\in P[2]}\int_P \bigl(e^{it\Psi(z)}-1\bigr) \psi_{(1,0,2)}(z) \,dx\,dy
\end{align*}
for small $t$, where we recall that the period function $\Psi$ is given by
\begin{align*}
\Psi(z)=\frac{2}{\sqrt{\vert D\vert}}\Im\left(\boar{\frac{1}{z}}-\boar{\frac{1}{\Tgaus(z)}}\right).
\end{align*}

The goal of this section is to establish the following asymptotic expansion.
\begin{prop}\label{prop:evaloftheintegral}
As $t\to 0$, the following asymptotic expansion holds:
\begin{align*}
\int_{I_D} \bigl(e^{it\Psi(z)}-1\bigr) \psi_{(1,0,2)}(z) \,dx\,dy = C_D t^2 \ln \vert t\vert + C_D' t^2 +  \LandauO\bigl(t^{\frac{5}{2}}\bigr),
\end{align*}
where $C_D, C_D'$ are real constants, with $C_D>0$, depending on $D\in \{2, 7, 11\}.$
\end{prop}
The proof is given at the end of this section, after the necessary preparation.

\subsection{Some facts concerning the measure \texorpdfstring{$\mu$}{mu} and the partition \texorpdfstring{$\Mpartition$}{P}}\label{subsec:measure}

Recall that the invariant measure $\mu$, whose density is $\psi_{(1,0,2)}(z)$, is, by construction, a probability measure that is absolutely continuous with respect to the Lebesgue measure. 
Given $z \in I$, we define the set
\begin{align}\label{eq:fractalset}
V(z)^\dagger \;:=\; 
\overline{
\bigcup_{n \geq 1}
\Bigl\{ -\frac{Q_n(z_n)}{Q_{n-1}(z_n)} \;\Big\vert\; z_n \in I,\; \Tgaus^n(z_n) = z \Bigr\}
}
\;\cup\; \{\infty\},
\end{align}
where $Q_n(z)$ denotes the denominator of the $n$-th convergent to $z$, and the overline denotes the closure of the set. 
It can be proved (see \cite[Proposition 5.3]{Nakada}) that if $z\in P$ for some $P\in\Mpartition[2]$, then $V(z)^\dagger$ depends only on the element $P\in \Mpartition$ containing $z$, not on $z$ itself. 
Then, for $z\in P$, the density for the measure $\mu$ is given by 
\begin{align*}
\psi_{(1,0,2)}(z)=\int_{P^\dagger} \frac{1}{\vert z-w\vert^4} \,dw.
\end{align*}
where $P^\dagger=V^\dagger(z)$.
This measure is absolutely continuous with respect to the Lebesgue measure and defines a probability measure on $I_D$ (see \cite[Sections 3--4]{Kimetall} and \cite[Sections 4--5]{Nakada}). 

Let us now define the functions
\begin{align}
(\psi_R(z), \psi_I(z)) = 
\begin{cases}
\displaystyle \left( \Re\boar{\frac{1}{z}},\; \dfrac{1}{\sqrt{D}} \Im\boar{\frac{1}{z}} \right), & \text{if } D = 2, \\[4pt]
\displaystyle \left( \Re\boar{\frac{1}{z}},\; \dfrac{2}{\sqrt{D}} \Im\boar{\frac{1}{z}} \right), & \text{if } D = 7 \text{ or } D = 11.
\end{cases}
\label{eq:psi_def}
\end{align}
We aim to understand the sets of points $z$ for which $\psi_I(z)=n$ for some $n\in \mathbb{Z}$. 
Therefore, we define the level sets
\begin{align}\label{defn:Vn}
V_n = \{ z \in I_2 : \psi_I(z) = n \},
\end{align}
and decompose them further as
\begin{align}\label{defn:Vrn}
V_{r,n} = \begin{cases}
    \{ z \in I_2 : \psi_I(z) = n,\ \psi_R(z) = r \} & n,r\in \mathbb{Z}, \quad D=2, \\
    \{ z \in I_2 : \psi_I(z) = n,\ \psi_R(z) = r+ \frac{n}{2} \} & n,r\in \mathbb{Z}, \quad D=7,11.
\end{cases}
\end{align}
Observe that for any $(m,n)$ and $r \neq r'$, the interiors satisfy $\operatorname{Int}(V_{r,n}\cap V_{r',m})=\emptyset$. 
We now define
\begin{align}\label{eq:Vnpmsplit}
V^+_n=\bigcup_{r>0} V_{r,n},\qquad V^{-}_n=\bigcup_{r<0} V_{r,n}.
\end{align}
It is easy to see that $V_n=V_{0,n}\cup V^{-}_n\cup V^+_n$.

\subsection{Restricting to relevant parts}
We now have to understand the volumes of the sets $V_{r,n}$ defined above. To this end, it will be important for us to understand the parts $P\in \Mpartition[2]$ that have the origin as a boundary point. Until now, no important case by case distinction upon $D$ was necessary. However, we must now understand how the sets $V_{r,n}$ split among the different pieces $P\in \Mpartition$. In the following, the analysis for the case $D=2$ is simpler than in the other two cases. 
\begin{lem}[The rectangular case: $D=2$]\label{lem:central-split-D-2}
    There is an $n_0\geq 1$ and a list $P_{1 \leq i \leq 8}$ of elements in $\Mpartition[2]$ such that for $\max(|r|, \vert n\vert) \geq n_0$, the interior of the sets $V_{r,n}$ is contained in exactly one of these two-dimensional parts. If these parts are labeled as in \cref{fig:D=2-central-split}, then all $(P_i)_{i\leq 4}$ have the same $\mu$-measure $C_2$, where
    \begin{align*}
        C_2 := \mu(P_1) = \int_{I_2} \left(\int_{P_1^\dagger} \frac{\chi_{P_1}(z)}{\vert z - w \vert^4} \, dw\right)\, dx\,dy.
    \end{align*}
    More precisely, for positive integers $r,n$ with $\max(r,n)>n_0$, we have $V_{r,-n} \subseteq P_1$, $V_{-r,-n} \subseteq P_2$, $V_{-r,n} \subseteq P_3$, and $V_{r,n} \subseteq P_4$. Moreover, for all $r,n$ with $n \geq n_0$, we have $\mu(V_{r,n}) = \mu(V_{|r|, |n|})$. For the remaining four parts, we have, for $n>n_0$, $V_{n,0}\subset P_5$, $V_{0, -n}\subset P_6$, $V_{-n,0}\subset P_6$ and $V_{ 0, n}\subset P_8$. We also have $\mu(P_5)=\mu(P_7)$ and $\mu(P_6)=\mu(P_8)$.
\end{lem}

\begin{proof}
The sets $V_{r,n}$ are the images of $I_2$ under the map $h_\alpha(z)$ with $\alpha=r+in\sqrt{2}$. Thus $V_{r,n}$ is contained in a disk of radius $\frac{1}{\sqrt{r^2+2n^2}}$. Moreover, the map is proper on a disk containing $I_2$, thus the boundary of $V_{r,n}$ is $f(\partial I_2)$. For positive $r,n$ with $\max(r,n)>n_0$, a trivial calculation shows that the $V_{r,n}$ are contained in $P_4$ (see \cref{fig:D=2-central-split}). The allocation of the other cases follows in the same way. For the equality of measures: using the definition in \eqref{eq:fractalset}, we see that for $V_{r,n}\subset P_1$
\[
\mu(V_{r,n})=\int_{I_2}\left(\int_{V(z)^\dagger} \frac{\chi_{V_{r,n}} (z)}{\vert z-w\vert^4} \,dw\right) dx\, dy= \int_{I_2} \left(\int_{P_1^\dagger} \frac{\chi_{V_{r,n}} (z)}{\vert z-w\vert^4} \,dw\right) dx\,dy.
\]
Now, notice that $\Tgaus(-z)=-\Tgaus(z)$ and $\Tgaus(\bar{z})=\overline{\Tgaus(z)}$. Then, $V(z)^\dagger=-V(-z)^\dagger$ and $\overline{V(z)^\dagger}=V(\bar{z})^\dagger$, from which we deduce $\mu(V_{r,n})=\mu(V_{\vert r\vert, \vert n\vert})$. The other equalities follows along the same lines. 
\end{proof}

\begin{rem}\label{rem:symmetry-remark}
    For later use, notice that given any $n\not=0$, $\psi_I$ is \textit{symmetric} for the level sets $V_n$ above, that is, we have $\mu(V_n)=\mu(V_{-n})$. To see this, assume without loss of generality that $V_n$ is fully contained inside a single $P\in \Mpartition$ (otherwise decompose $V_n$ among the finitely many parts $P_i\in \Mpartition[2]$ it intersects with), use the definition of $V(z)^\dagger$ in \eqref{eq:fractalset} and proceed as in the above proof.
\end{rem}
\begin{figure}[tp]
    \centering
    \includegraphics[width=0.45\linewidth]{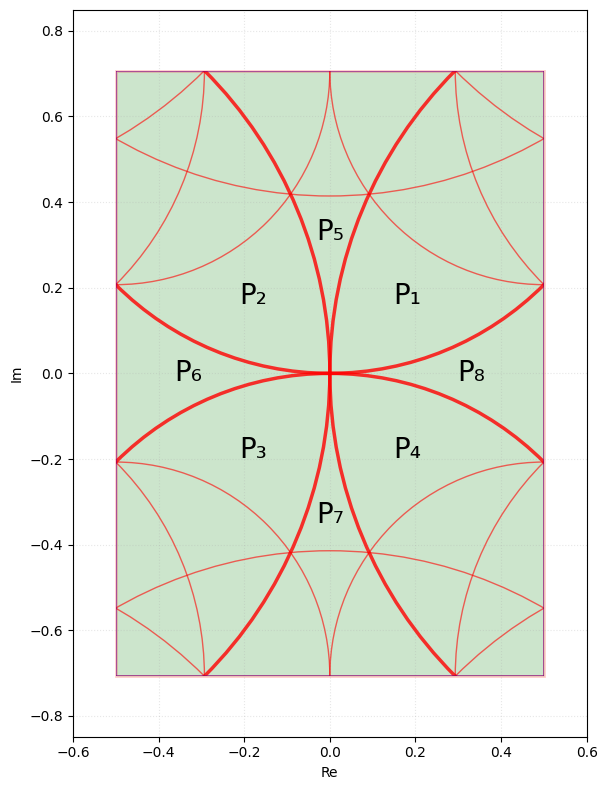}
    \caption{Labeling of the elements $(P_i)_{1\leq i\leq 8}$ in the $D=2$ case, as in \cref{lem:central-split-D-2}.}
    \label{fig:D=2-central-split}
\end{figure}
\begin{rem}\label{rem:remakron711tricky}
As we will show later, the contribution of the parts $V_{r,n}$ with $rn=0$ is negligible, thus only the parts $(P_i)_{1\leq i\leq 4}$ are relevant in the case $D=2$, and all of them have the same $\mu$-measure. This is not the case for $D=7$ or $D=11$. Indeed, as illustrated in \cref{fig:D=117-central-split}, in these cases there are two types of relevant parts: those labeled $(P_i)_{1\leq i\leq 4}$ and those labeled $(P_i)_{5\leq i\leq 6}$. Moreover, unlike in the $D=2$ case, the presence of many extra ``cuspidal'' parts makes a precise description of the allocation of $V_{r,n}$ into the parts $(P_i)_{i\leq 8}$ tedious. This will play a role in the proof of \cref{prop:phi_Icontinuation}. Indeed, we will first handle the case $D=2$ using the clean allocation of the sets $V_{r,n}$ from \cref{lem:central-split-D-2} and then  explain the necessary modifications for the remaining two cases.
\end{rem}

\begin{figure}[tp]
    \centering
    \includegraphics[width=1\linewidth]{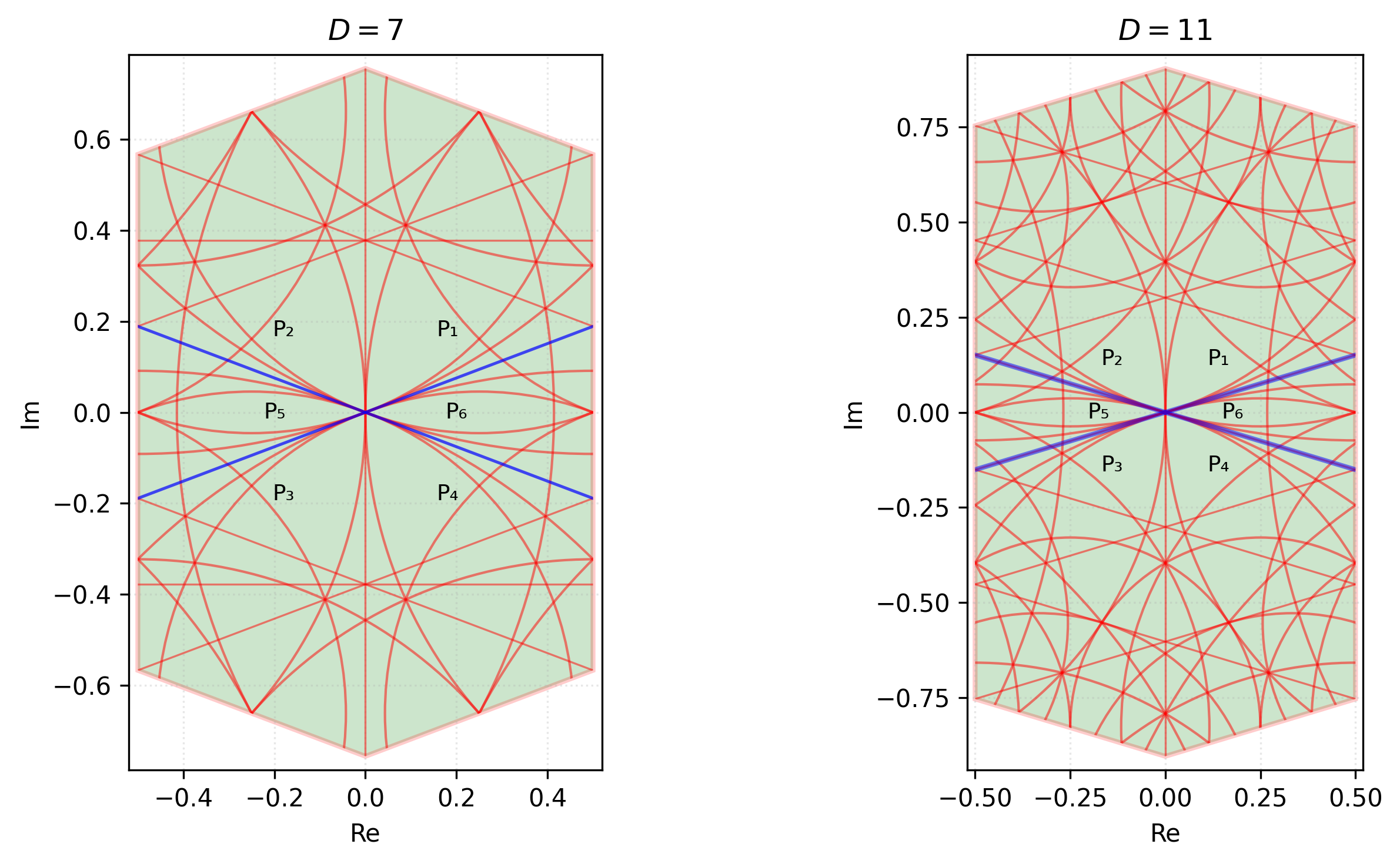}
    \caption{Central subdivisions into relevant parts $(P_i)_{1\leq i\leq 6}$, with auxiliary lines depicted in blue.}
    \label{fig:D=117-central-split}
\end{figure}

\subsection{Technical preliminaries}
For our purposes, we aim to compute the quantity
\[
\int_{I_D} \bigl(e^{it\psi(z)}-1\bigr) \psi_{(1,0,2)}(z) \,dx\,dy = \sum_P \int_P \bigl( e^{it\psi(z)} - 1 \bigr) \psi_{(1,0,2)}(z) \,dx\,dy,
\]
for real-valued functions $\psi: I_D \to \mathbb{R}$. Following \cite{Bettin-Drappeau-integral}, let $\eta\in [0,1]$ be some small parameter. Define then the quantity
\begin{align}\label{eq:Feta}
    F(\eta, s):=\int_{I_D} w_{(\eta, s)}(z) \,d\mu(z)=\sum_P \int_P w_{\eta, s}(z) \psi_{(1,0,2)}(z) \,dx\,dy, 
\end{align}
where
\begin{align*}
w_{\eta, s}(z):=\mathbf{1}_{\psi(z)\neq 0} \, \vert \psi(z)\vert^s \exp\left(-s\frac{\pi i}{2}(1-\eta)\sgn(\psi(z))\right).    
\end{align*}
Furthermore, let $c>0$ be a real number for which the integral
\begin{align}\label{eq:c-admissible}
    \int_{I_D} \bigl( \vert \psi(z)\vert^c+\vert \psi(z)\vert^{-c}\bigr) \mathbf{1}_{\psi(z)\neq 0} \,d\mu(z) 
\end{align}
is finite. We call any such $c$ \textit{admissible}.

\begin{lem}\label{lem:2admissible}
Let $\psi=\psi_I$. Then any positive value $c<2$ is admissible in \eqref{eq:c-admissible}.
\end{lem}

\begin{proof}
By definition, regardless of $D$, $\psi$ takes only integer values. Let $|n|\geq n_D$ for some sufficiently large $n_D$. The contribution of all $z\in I_D$ for which $\psi(z)\leq n_D$ is bounded (since $\mu(z)$ is a probability density on $I_D$) and therefore does not affect the convergence of the sum. Assume now that $D=2$. For any $(r,n)$ with $|n| > n_2$, the sets $V_{r, n}$ are the image of $I_2$ under the map $z\mapsto \frac{1}{z+\alpha}$ for $\alpha = r + in\sqrt{2}$. Now, since $\mu(z)$ is absolutely continuous with respect to the Lebesgue measure, we have
\[
\mu(V_{r,n})\ll \operatorname{vol}(V_{r,n}) = \int_{V_{r,n}} 1 \,dx\,dy = \int_{I_2}\frac{1}{\vert z + \alpha\vert^4} \,dx\,dy \ll \frac{1}{\vert \alpha\vert^4}=\frac{1}{(n^2+2r^2)^2}.
\]
Thus, we conclude
\[
\int_{I_2} \vert \psi_I(z)\vert^c \,d\mu(z) \ll \sum_{|n| \geq n_2} |n|^c \, \mu(V_n) \ll \sum_{n\geq n_2} \frac{1}{n^{1+\epsilon}}<\infty.
\]
The proof in the cases $D=7, 11$ follows along the same lines.
\end{proof}

Let now $\psi$ be a function $I_D\to \RR$ with some admissible $c>0$. Define, for $\epsilon\in [0, \eta']$, $\eta'>0$, the quantities
\begin{align*}
J_+(\epsilon)=\int_{I_+} e^{(-\epsilon+i)t\psi(z)} \,d\mu(z),\qquad J_-(\epsilon)=\int_{I_{-}} e^{(\epsilon+i)t\psi(z)} \,d\mu(z),
\end{align*}
where  $I_{-}:=\{z\in I_D: \psi(z)<0\}$, $I_0:=\{z\in I_D: \psi_I(z)=0\}$ and $I_{+}:=\{z\in I_D: \psi(z)>0\}$. Then we have the following technical proposition.
\begin{prop}\label{prop:integral-first-step}
Let $H(2, 2+\delta)$ be a Hankel contour (see e.g. \cite{Tenenbaum}) around the point $2$, with thickness $\delta<1$, tracing around $2$ clockwise, and let $F(\eta,s)$ be given as in \eqref{eq:Feta}. Then, for $\psi=\psi_I$, we have
\begin{align*}
J_+(\varepsilon)+J_{-}(\varepsilon)+\mu(\{x\in I_D: \psi_I(x)=0\})&= 1 + F(\eta, 1)t\\
&+\frac{1}{2\pi i} \int_{H(2,2+\delta)} \Gamma(-s)F(\eta, s)t^{s}|1 + i\varepsilon|^{s} \,ds \\
&+  \frac{1}{2\pi i}\int_{\mathrm{Re}(s)=2+\delta} \Gamma(-s)F(\eta, s)t^{s}|1 + i\varepsilon|^{s} \,ds,
\end{align*}
where $\eta=\arctan (\varepsilon)\leq \varepsilon\leq \eta'$.
\end{prop}

\begin{proof}
By \cite[Eq 17.43.1, Eq 8.327.1]{Grad-Riz-Tables}, we have
\begin{align*}
    e^{-y}=\frac{1}{2\pi i}\int_{-c/2-i\infty}^{-c/2+i\infty} \Gamma(-s) |y|^s e^{s \arg(y)} \,ds,
\end{align*}
provided $\Re(y)>0$. With $y=t\psi(z)(i-\varepsilon)$, we get
\begin{align*}
J_+(\varepsilon)+J_-(\varepsilon)=\frac{1}{2\pi i}\int_{-c/2-i\infty}^{-c/2+i\infty} \Gamma(-s) F(\eta, s) |1+i\varepsilon|^s t^s \,ds,
\end{align*}
where $\eta = \arctan \varepsilon \leq \varepsilon \leq \eta_{0}$. 
Now we shift the contour forward to $\mathrm{Re}(s) = 2 + \delta$, obtaining
\begin{align}\label{eq:contourshift}
J_+(\varepsilon) + J_-(\varepsilon) = -(R_0+R_1) &+ \frac{1}{2\pi i} \int_{H(2,2+\delta)} \Gamma(-s)F(\eta, s)t^{s}|1 + i\varepsilon|^{s} \,ds \\
&+ \frac{1}{2\pi i} \int_{\mathrm{Re}(s)=2+\delta} \Gamma(-s)F(\eta,s)t^{s}|1 + i\varepsilon|^{s} \,ds,\nonumber
\end{align}
where $R_0$ and $R_1$ are the residues at $s=0$ and $s=1$. Using the Taylor expansion at $s=0$,
\[
\Gamma(-s) = -\frac{1}{s} - \gamma +\LandauO(s),
\]
we get 
\[
R_0=-\mu\bigl( \{ z\in I: \psi_I(z)\neq 0 \} \bigr).
\] 
Similarly, using the expansion at $s=1$,
\[
\Gamma(-s) = \frac{1}{s-1} + (\gamma - 1) + \LandauO(s-1),
\]
we see that the residue at $s=1$ is given by $F(\eta, 1)t$. Now, notice that since $\mu(I_D)=1$, we have
\begin{align*}
J_+(\varepsilon)+J_{-}(\varepsilon)+\mu(\{x\in I_D: \psi_I(x)=0\})&= 1 + F(\eta, 1)t\\
&+ \frac{1}{2\pi i} \int_{H(2,2+\delta)} \Gamma(-s)F(\eta, s)t^{s}|1 + i\varepsilon|^{s} \,ds \\
&+ \frac{1}{2\pi i} \int_{\mathrm{Re}(s)=2+\delta} \Gamma(-s)F(\eta, s)t^{s}|1 + i\varepsilon|^{s} \,ds,
\end{align*}
as claimed.
\end{proof}

\subsection{Concluding the proof of Proposition~\ref{prop:evaloftheintegral}}
In order to evaluate the two integrals in \cref{prop:integral-first-step}, we must understand the analytic continuation of $F(\eta, s)$. We show that $F(0, s)$ has a simple pole at $s = 2$. To prove this, we need more precise control on the measure $\mu$, as absolute continuity with respect to the Lebesgue measure is not sufficient.

\begin{lem}\label{lem:asymptotic-measure}
Let $D\in \{2,7,11\}$ be fixed and let $U\subset P$ for $P\in \mathcal{P}[2]$. Assume that the diameter of the set $U$ satisfies $\operatorname{diam}(U)<1/2$. Then we have
\[
\mu(U) = \kappa_{P, D} \,\operatorname{vol}(U) \bigl(1 + \LandauO(\operatorname{diam}(U))\bigr),
\]
where $\kappa_{P,D}= \int_{P^\dagger} \frac{1}{\vert w\vert^4} \,dw$ is a constant depending only on $P$.
\end{lem}

\begin{proof}
Since $U\subset P$, we have 
\[
\mu(U)=\int_{U} 1 \cdot \psi_{(1,0,2)}(z) \,dx\,dy,
\]
where the density is given by $\psi_{(1,0,2)}(z)=\int_{P^\dagger} \frac{1}{\vert z-w\vert^4} \,dw$. We now use Taylor expansion on the integrand $\frac{1}{|z-w|^4}$ appearing in the density around $z=0$:
\[
\frac{1}{|z-w|^4} = \frac{1}{|w|^4} + \frac{2}{|w|^4} \left( \frac{z}{w} + \frac{\bar{z}}{\bar{w}} \right) + \sum_{\substack{n,m \ge 0 \\ n+m \ge 2}} \frac{(n+1)(m+1)}{w^{n+2} \bar{w}^{m+2}} z^n \bar{z}^m.
\]
Now, since 
\[
\frac{4}{|w|^4} \Re\left( \frac{z}{w} \right) = \LandauO\left(\frac{\vert z\vert}{\vert w\vert^5}\right)
\]
and 
\[
\left\vert \sum_{\substack{n,m \ge 0 \\ n+m \ge 2}} \frac{(n+1)(m+1)}{w^{n+2} \bar{w}^{m+2}} z^n \bar{z}^m \right\vert
\ll \frac{1}{\vert w\vert^4} \sum_{r=2}^\infty (r+1)^3 \left(\frac{\vert z\vert}{\vert w\vert}\right)^r
\ll \frac{\vert z\vert^2}{\vert w\vert^6}
\quad \text{for } \vert z\vert < \vert w\vert,
\]
we obtain 
\[
\psi_{(1,0,2)}(z) = \int_{P^\dagger} \frac{1}{\vert w\vert^4} \bigl(1 + \LandauO(\vert z\vert)\bigr) \,dw = \kappa_P + \LandauO(\vert z\vert).
\]
Therefore,
\[
\mu(U) = \int_{U} \bigl(\kappa_P + \LandauO(\operatorname{diam}(U))\bigr) \,dx\,dy = \kappa_P \, \operatorname{vol}(U) \bigl(1 + \LandauO(\operatorname{diam}(U))\bigr).
\]
\end{proof}

The following lemma gives us an asymptotic evaluation of the measure of the level sets $V_n$ for the case $D=2$. 

\begin{lem}[The measure of the level sets $V_n$]\label{lem:measureVn}
Let $D\in \{2,7, 11\}$ and let $n\geq n_D$. Then, for the measure of the sets $V_n$ in \eqref{defn:Vn}, we have:
\[
\mu(V_{n})= \kappa_{P,2}\frac{\pi}{4\sqrt{2} n^3}+ \LandauO\left(\frac{1}{n^4}\right),
\]
where $\kappa_{P,2}$ is as in \cref{lem:asymptotic-measure}.
\end{lem}

\begin{proof}
Let $n\geq n_0$, for $n_0$ defined as in \cref{lem:central-split-D-2}. By the same lemma, all $V_{r,n}$ with $nr\neq 0$ lie fully in some $P_i$ for $i\in \{1,\dots, 4\}$. We also have $\mu(V_n)=\mu(V_n^{-}\cup V_n^+)+\LandauO\left(\frac{1}{n^4}\right)$ uniformly. By \cref{lem:asymptotic-measure} for $D=2$ we have
\[
\mu(V_{r, n})=\kappa_{P,2} \operatorname{vol}(V_{r,n})\bigl(1+\LandauO(\operatorname{diam}(V_{r,n}))\bigr),
\]
with $\operatorname{diam}(V_{r,n})\ll \frac{1}{|r+in\sqrt{2}|}$. Now, writing $\alpha=r+in\sqrt{2}$, we have
\[
\operatorname{vol}(V_{r,n})=\int_{I} \frac{1}{|z+\alpha|^4} \,dx\,dy=\frac{1}{|\alpha|^4}+ \LandauO\left(\frac{1}{|\alpha|^8}\right).
\]
Thus,
\[
\mu(V_n)=\sum_{r\in \mathbb{Z}} \frac{\kappa_{P, 2}}{(r^2+2n^2)^2}\left(1+\LandauO\left(\frac{1}{|n|}\right)\right).
\]
Now, by Poisson summation,
\begin{align*}
\sum_{r \in \mathbb{Z}} \frac{1}{(r^2 + 2n^2)^2} 
&= \int_{-\infty}^{\infty} \frac{1}{(x^2 + 2n^2)^2}\, dx + \sum_{k \neq 0} \int_{-\infty}^{\infty} \frac{e^{2\pi i k x}}{(x^2 + 2n^2)^2}\, dx \\
&= \frac{\pi}{4\sqrt{2} n^3} + \LandauO\left(\frac{1}{n^6}\right).
\end{align*}
Therefore, for $|n|>n_0$ we obtain
\[
\mu(V_{n})= \kappa_{P,2}\frac{\pi}{4\sqrt{2} n^3}+ \LandauO\left(\frac{1}{n^4}\right).
\]
\end{proof}

\begin{rem}
 As mentioned in \cref{rem:remakron711tricky}, for the cases $D=7, 11$, we encounter some extra complications, which makes a clean description of the allocation of the sets $V_{r,n}$ rather difficult. In order to streamline our argument, we introduce the following lemma, where $A(n)$ is essentially $\mu(V_n)$ in disguise.
\end{rem}

\begin{lem}[Asymptotics for $V_n$-sets in the $D=7,11$]\label{lem:711Vnmeasure}
Let $D\in \{7, 11\}$, let $n\geq n_D$. Define the quantity
\begin{align*}
A(n) = \kappa_{P_1, D} \sum_{\substack{r\in\mathbb{Z} \\ \left| \frac{n}{2r+n} \right| \ge \frac{1}{D}}} 
\frac{1}{\bigl[(r + n/2)^2 + n^2 D/4\bigr]^2}
+ \kappa_{P_6, D} \sum_{\substack{r \in \mathbb{Z} \\ -\frac{1}{D} \le \frac{n}{2r+n} \le \frac{1}{D}}} 
\frac{1}{\bigl[(r + n/2)^2 + n^2 D/4\bigr]^2}.
\end{align*}
Then, there is some $C_D>0$ such that
\begin{align*}
A(n)=\frac{C_D}{n^3}+ \LandauO\left(\frac{1}{n^6}\right).
\end{align*}
\end{lem}

\begin{proof}
By Poisson summation, we have
\begin{align*}
\sum_{\substack{r \in \mathbb{Z} \\ |2r + n| \ge nD}} 
\frac{1}{\bigl[(r + n/2)^2 + n^2 D/4\bigr]^2}
&= \frac{1}{n^3} \int_{|u| \ge D/2} \frac{du}{(u^2 + D/4)^2} \;+\; \LandauO\left(\frac{1}{n^6}\right) \\
&= \frac{1}{n^3} \left( \frac{4\pi - 8\arctan\sqrt{D}}{D^{3/2}} - \frac{8}{D(D+1)} \right) + \LandauO\left(\frac{1}{n^6}\right).
\end{align*}
Similarly,
\begin{align*}
\sum_{\substack{r\in\mathbb{Z} \\ |2r+n|\le nD}} 
\frac{1}{\bigl((r+n/2)^2 + n^2 D/4\bigr)^2}
&= \frac{1}{n^3} \int_{-D/2}^{D/2} \frac{du}{(u^2 + D/4)^2} + \LandauO\left(\frac{1}{n^6}\right)\\
&=\frac{1}{n^3} \left(\frac{8}{D(D+1)} + \frac{8}{D^{3/2}} \arctan\sqrt{D}\right) + \LandauO\left(\frac{1}{n^6}\right).
\end{align*}
So we conclude that 
\begin{align*}
    A(n)=\frac{C_D}{n^3}+ \LandauO\left(\frac{1}{n^6}\right),
\end{align*}
for some $C_D>0$.
\end{proof}

\begin{lem}[Analytic continuation of the function $F$]\label{prop:phi_Icontinuation}
Let $\psi=\psi_I$. Then the function $F(\eta, s)$ in \eqref{eq:Feta} is analytic for any fixed $\eta$ and $\Re(s)<3$, with the sole exception of a simple pole at $s=2$. In particular, there exist a $C_D>0$ such that we have
\[
F(\eta, s)=h(s)+ C_D\cos\left(\frac{s\pi}{2} (1-\eta)\right) \zeta(3-s),
\]
for a function $h(s)$ analytic for $\Re(s)<3$. $F(0, s)$ is therefore analytic in $\Re(s)<3$ with the exception of a simple pole at $s=2$.
\end{lem}

\begin{proof}
We begin with the $D=2$ case. We have
\begin{align*}
F(\eta, s)&=\exp\left(-s\frac{\pi i}{2}(1-\eta)\right) \sum_{n\geq 1} n^s \times \mu(V_n)\\
&\quad + \exp\left(s\frac{\pi i}{2}(1-\eta)\right) \sum_{n\geq 1} n^s \times \mu(V_{-n}),
\end{align*}
provided $\Re(s)<2$. Since omitting a finite number of terms in the sums above does not alter the analytic properties of $F(\eta, s)$, we can restrict ourselves to $|n|>n_0$ for some large $n_0$. Using now \cref{lem:measureVn} we obtain
\begin{align*}
F(\eta, s)&=\exp\left(-s\frac{\pi i}{2}(1-\eta)\right) \sum_{n\geq 1} n^s \mu(V_n)+ \exp\left(s\frac{\pi i}{2}(1-\eta)\right) \sum_{n\geq 1} n^s\mu(V_{-n})\\
&=H_1(s)+ \left(\exp\left(-s\frac{\pi i}{2}(1-\eta)\right)+\exp\left(s\frac{\pi i}{2}(1-\eta)\right)\right) \frac{\kappa_{P,2}\pi}{4\sqrt{2}} \sum_{n\geq n_0} n^{s-3}\\
&=H_2(s)+ \frac{\kappa_{P,2}\pi}{4\sqrt{2}} \left(\exp\left(-s\frac{\pi i}{2}(1-\eta)\right)+\exp\left(s\frac{\pi i}{2}(1-\eta)\right)\right)  \zeta(3-s)\\
&=H_2(s)+\frac{\kappa_{P,2}\pi}{2\sqrt{2}} \cos\left(\frac{s\pi}{2} (1-\eta)\right) \zeta(3-s),
\end{align*}
where the functions $H_1(s)$ and $H_2(s)$ account for the finite number of terms we omitted and the error term in the expansion for $\mu(V_n)$. In particular, both are analytic for $\Re(s)<3$. Therefore, $F(0, s)$ is a meromorphic function in the half-plane $\Re(s)<3$, with a simple pole at $s=2$.
\\
\par
Now we move to the cases $D \in \{7, 11\}$. The caveat here is that a given cell $V_{r,n}$ may intersect both a region $P \in \{ P_1, \dots, P_4 \}$ and a region $P \in \{ P_5, P_6 \}$ simultaneously. This could affect the application of \cref{lem:asymptotic-measure} as done above via \cref{lem:measureVn}. Since we do not wish to classify which $V_{r,n}$ falls into which of the many regions $P \in \Mpartition[2]$ with $0 \in \partial P$, we show that it suffices to consider only the pairs $(r,n)$ for which $V_{r,n}$ lies in one of the four regions determined by cutting $I_D$ along the lines of slope $\pm 1/\sqrt{D}$, weighting them according to the constants $\kappa_{P_1, D}$ and $\kappa_{P_6, D}$, in the sense of \cref{lem:asymptotic-measure}, see \cref{fig:D=117-central-split} for a visual support. This allows us to replace $\mu(V_n)$ with the $A(n)$ from \cref{lem:711Vnmeasure}.

Let us consider the three lines $l_1, l_2, l_3$ given by $l_1: y=-\frac{1}{\sqrt{D}} x$, $l_2: y=\frac{1}{\sqrt{D}} x$, and $l_3: x=0$ (again, use \cref{fig:D=117-central-split} for visual support). These lines are the tangents to the opposing circles whose arcs define the parts $P_1,\dots,P_6$. Notice that $P_1$ and $P_3$ lie between $l_2$ and $l_3$, and $P_2, P_4$ lie between $l_1$ and $l_3$. Similarly, $P_5$ and $P_6$ lie between $l_1$ and $l_3$. Let now $\alpha=r+n\left(\frac{1+\sqrt{-D}}{2}\right)$. Then, for $n\geq n_D$, we have $V_{r,n}=h_\alpha(I)$ and
\[
V_n=\bigcup_{r\in \mathbb{Z}} V_{r,n}.
\]
Note that (the interior of) $V_{r,n}$ may not be contained in a single $P\in \Mpartition$. Our first step is to bound the contribution from those $V_{r,n}$ not fully contained in a single $P_1,\dots,P_6$. Let us define the complement $B:= B_{1/|n_D|}(0)\setminus \bigcup_{i\leq 6} P_i$ and 
\[
a_n:=\sum_{\substack{r\in \mathbb{Z}\\ V_{r,n}\cap B \neq \emptyset}} \mu(V_{r,n}\cap B).
\]
Then we consider the sum
\begin{align*}
    H_3(s)=\sum_{n\geq n_D} a_n n^s.
\end{align*}
By the absolute continuity of the measure $\mu$ with the Lebesgue measure, a comparison with the lunar area integral
\begin{align*}
    \int_0^{1/n_D} \bigl(R-\sqrt{R^2-t^2}\bigr) t^{s} \,dt,
\end{align*}
where $R$ is the radius of the circles meeting the origin in \cref{fig:D=117-central-split}, shows that the series above converges absolutely and uniformly for $\Re(s)<3$, so $H_3(s)$ is analytic for $\Re(s)\leq 3$. A similar argument works for the contribution from those $V_{r,n}$ (partially) contained in some cuspidal $P_i$ for some $P_i\in \Mpartition$ (see \cref{fig:D=117-central-split}). Therefore, up to modifying $H_3(s)$ slightly to $H_4(s)$, we can assume that all $V_{r,n}$ with $-\frac{1}{D}\leq \frac{n}{2r+n}\leq \frac{1}{D}$ lie inside $P_5$ or $P_6$ and that all the others lie inside one of the four remaining $P_i$. 
Hence, using \cref{lem:asymptotic-measure}, we conclude 
\begin{align*}
F(\eta, s)&=\exp\left(-s\frac{\pi i}{2}(1-\eta)\right) \sum_{n\geq 1} n^{s} \mu(V_n) + \exp\left(s\frac{\pi i}{2}(1-\eta)\right) \sum_{n\geq 1} n^{s} \mu(V_{-n})\\
&=H_4(s)+ \left(\exp\left(-s\frac{\pi i}{2}(1-\eta)\right)+\exp\left(s\frac{\pi i}{2}(1-\eta)\right)\right)  \sum_{n>n_D} n^{s} A(n),
\end{align*}
where
\begin{align*}
A(n) = \kappa_{P_1, D} \sum_{\substack{r\in\mathbb{Z} \\ \left| \frac{n}{2r+n} \right| \ge \frac{1}{D}}} 
\frac{1}{\bigl[(r + n/2)^2 + n^2 D/4\bigr]^2}
+ \kappa_{P_6, D} \sum_{\substack{r \in \mathbb{Z} \\ -\frac{1}{D} \le \frac{n}{2r+n} \le \frac{1}{D}}} 
\frac{1}{\bigl[(r + n/2)^2 + n^2 D/4\bigr]^2}.
\end{align*}
Thus, by an application of \cref{lem:711Vnmeasure}, we have
\begin{align*}
F(\eta, s)&=\exp\left(-s\frac{\pi i}{2}(1-\eta)\right) \sum_{n\geq 1} n^{s} \mu(V_n) + \exp\left(s\frac{\pi i}{2}(1-\eta)\right) \sum_{n\geq 1} n^{s} \mu(V_{-n})\\
&=H_4(s)+ 2\cos\left(\frac{s\pi}{2}(1-\eta)\right) C_D  \sum_{n>n_D} n^{s-3},
\end{align*}
for some $C_D>0$, which concludes the proof.
\end{proof}

\begin{prop}[The evaluation of the integrals]\label{prop:evalintegralfinal}
There exist real constants $C_D, C_D'$, with $C_D>0$, such that the following asymptotic expansion 
    \begin{align*}
\int_I \left(e^{it\psi_I(z)}-1\right) d\mu(z)&= C_D \, t^2 \ln \vert t\vert + C'_D t^2 + \LandauO\left(t^{\frac{5}{2}}\right),
\end{align*}
holds.
\end{prop}

\begin{proof}
By \cref{prop:integral-first-step}, we have
\begin{align}
J_+(\varepsilon)+J_{-}(\varepsilon)+\mu(\{x\in I_D: \psi_I(x)=0\})&= 1 + F(\eta, 1)t  \notag \\
&+\frac{1}{2\pi i} \int_{H(2,2+\delta)} \Gamma(-s)F(\eta, s)t^{s}|1 + i\varepsilon|^{s} \,ds \label{eq:someq1} \\
&+  \frac{1}{2\pi i}\int_{\mathrm{Re}(s)=2+\delta} \Gamma(-s)F(\eta, s)t^{s}|1 + i\varepsilon|^{s} \,ds. \label{eq:someq2}
\end{align}
For $\Re(s)<2$ we have the representation of $F(\eta, s)$ as 
\begin{align*}
F(\eta, s)&=\int_{I_{+}}|\psi(z)|^s \exp\left(-s\frac{\pi i}{2}(1-\eta)\right) d\mu(z)+ \int_{I_{-}}|\psi(z)|^s \exp\left(s\frac{\pi i}{2}(1-\eta)\right) d\mu(z)\\
&= \exp\left(-s\frac{\pi i}{2}(1-\eta)\right) \sum_{n\geq 1} n^s \times \mu\bigl(\{z\in I: \psi_I(z)=n\}\bigr)\\
&\quad+  \exp\left(s\frac{\pi i}{2}(1-\eta)\right) \sum_{n\geq 1} n^s \times \mu\bigl(\{z\in I: \psi_I(z)=-n\}\bigr).
\end{align*}
Trivially, we have 
\begin{align*}
    \int_{I_D} e^{it\psi_I(z)} d\mu(z)= \left(\int_{I_0}+\int_{I_{-}}+\int_{I_{+}}\right)e^{it\psi_I(z)} d\mu(z).
\end{align*}
Notice that by dominated convergence, we have 
\begin{align*}
    \lim_{\varepsilon\to 0^+} J_{\pm}(\varepsilon)= \int_{I_\pm} e^{it\psi_I(z)} d\mu(z).
\end{align*}
Therefore, we only need to handle $F(\eta,s)$ for $\eta=\arctan(\varepsilon)\leq \varepsilon\leq \eta'$. By \cref{rem:symmetry-remark} and since $\exp\left(\pm \frac{\pi i}{2}\right)=\pm i$, we conclude $F(0, 1)=0$. Select now $\delta=\frac{1}{2}$.  By the Stirling approximation for large $\tau$ and the standard bound $\zeta(\frac{1}{2}+i\tau)\ll (1+\tau)^{\frac{1}{6}}$ (see e.g \cite{Montgomery-Vaughan-Multnumbertheory, Grad-Riz-Tables}), we have
\[
|\Gamma(-s) F(\eta, s)| \ll |\tau|^{-\frac{17}{6} + \varepsilon} e^{-\frac{\pi}{2}\eta |\tau|}, \qquad \tau \to \pm\infty.
\]
So the second integral in \eqref{eq:someq2} is $\ll t^{\frac{5}{2}}$, uniformly in $\eta\geq 0$.
By \cref{prop:phi_Icontinuation} and the residue theorem for \eqref{eq:someq1}, we find
\begin{align*}
\operatorname{Res}_{s=2} \left[ \Gamma(-s) t^s F(0,s) \right]&= \frac{t^2}{2} \left[ -h(2) + \frac{3}{2} C_D - C_D \log t \right].
\end{align*}
Whence, passing to the limit $\epsilon\to 0^+$ we have
\begin{align*}
\int_I \left(e^{it\phi(z)}-1\right) d\mu(z)= \frac{C_D}{2} t^2 \log t + C_D't^2+ \LandauO\left(t^{\frac{5}{2}}\right),
\end{align*}
from which the statement follows upon adjusting the constants $C_D, C'_D$.
\end{proof}

We are now ready to prove \cref{prop:evaloftheintegral}.

\begin{proof}[Proof of \cref{prop:evaloftheintegral}]
We present the proof for the case $D=2$; the other two cases are nearly identical.
First, let $\alpha=r+in\sqrt{2}$. Then, recalling \eqref{defn:Vn} and absolute continuity of $\mu$, we have
\begin{align}\label{eq:intbound}
    \int_{I_2}\vert e^{it\psi_I(z)}-1\vert d\mu(z)\ll \sum_{\substack{n\in \ZZ} }\vert e^{itn}-1\vert \mu(V_n)\ll t.
\end{align}
Consider
\begin{align}\label{eq:twist}
    \int_{I_2} \bigl(e^{-it\psi_I(\Tgaus(z))}-1\bigr)\bigl(e^{it\psi_I(z)}-1\bigr) \,d\mu(z).
\end{align}
Recall that by \cref{rem:badalphas}, there are finitely many $\alpha\in \ringint$ such that $\Tgaus O_\alpha \not =I_2$. Therefore, we split the integral as
\begin{align*}
    \eqref{eq:twist}=\Bigl(\sum_{\substack{\alpha\in \ringint\\ \Tgaus O_\alpha \not = I_2}} + \sum_{\substack{\alpha\in \ringint\\ \Tgaus O_\alpha = I_2}}\Bigr)\int_{O_\alpha} \left(e^{-it\psi_I(\Tgaus(z))}-1\right)\left(e^{itn}-1\right) d\mu(z).
\end{align*}
We begin with the finite sum. Using $\vert e^{itn}-1\vert \ll \vert tn\vert$, the absolute continuity of $\mu$ the change of variables $z'=\Tgaus(z)$ and \eqref{eq:intbound}, we estimate
\begin{align*}
    \ll \sum_{\substack{\alpha\in \ringint\\ \Tgaus O_\alpha \not = I_2}} \vert tn\vert \int_{O_\alpha} \vert e^{-it\psi_I(\Tgaus(z))}-1\vert dxdy\ll \sum_{\substack{\alpha\in \ringint\\ \Tgaus O_\alpha \not = I_2}} \vert tn\vert \int_{I_2} \vert e^{-it\psi_I(z')}-1\vert dxdy\ll t^2.
\end{align*}
all the $\alpha$ in the second sum, we have $\Tgaus O_\alpha=I_2$. Therefore, taking the substitution $z'=\frac{1}{z}-\alpha$ and \eqref{eq:intbound}, we have
\begin{align*}
    \sum_{\substack{\alpha\in \ringint\\ \Tgaus O_\alpha = I_2}}\int_{O_\alpha} \left(e^{-it\psi_I(\Tgaus(z))}-1\right)\left(e^{itn}-1\right) d\mu(z)\ll \int_{I_2} \vert e^{it\psi_I(z)}-1\vert dxdy \times \sum_{n\geq 1} \frac{\vert e^{itn}-1\vert}{n^4}\ll t^2.
\end{align*}
Thu, we conclude that $\eqref{eq:twist}\ll t^2$.
Now notice that, by invariance of the measure together with \cref{prop:evalintegralfinal}, we have
\begin{align*}
    \int_{I_2} \bigl(e^{it\psi_I(\Tgaus(z))}-1\bigr) \,d\mu(z)
    = \int_{I_2} \bigl(e^{it\psi_I(z)}-1\bigr) \,d\mu(z)
    = C_D \, t^2 \ln|t| + C_D' t^2 + \mathcal{O}\!\left( t^{3-\epsilon} \right),
\end{align*}
and we therefore conclude that 
\begin{align*}
    \int_{I_2} \bigl(e^{it\Psi(z)}-1\bigr) \,d\mu(z) = C_D \, t^2 \ln|t| + \mathcal{O}\!\left( t^{2} \right).
\end{align*}
\end{proof}

\section{Proof of \cref{thm:maintheorem}}\label{sec:proof-of-the-theorem}
To conclude the proof of our main theorem, we require the following version of the Berry--Esseen inequality from \cite[Theorem 7.16]{Tenenbaum}.
\begin{lem}\label{lem:Berry-Eessen}
Let $F, G$ be two distribution functions with respective characteristic
functions $f, g$. Suppose that $G$ is differentiable and that $G'$ is bounded
on $\RR$. Then for all $T>0$,
\begin{equation*}
\|F-G\|_\infty
\le 16\frac{\|G'\|_\infty}{T} + 6\int_{-T}^{T}\frac{|f(\tau)-g(\tau)|}{|\tau|}\,d\tau.
\end{equation*}
\end{lem}

As in \cite{BettinDrappeau-main}, we encounter some difficulty in the estimating integral above when $\tau$ is very small. To this end, we need the following lemma, which essentially ensures the existence of a low moment for $S$ in \eqref{eq:S-var-definition}.
\begin{lem}\label{lem:veryverysmalt}
There exists a sufficiently large $L>0$, such that, for $|t|\le 1$,
\[
\ExpX\!\left[e^{it S(z)}\right]
=
1
+
\mathcal{O}\!\left( |t|\log^{L} X\right).
\]
\end{lem}
\begin{proof}
Let $z\in \FFFF$. Then, for $t$ sufficiently small, we have 
\begin{align*}
    \vert e^{itS(z)}-1\vert\leq \vert t \vert \sum_{j=1}^{r(z)} \vert \psi_j(\Tgaus^{j-1}(z)\vert.
\end{align*}
The statement then follows along the same lines as \cite{BettinDrappeau-main}, see also \cite[Lemma 3.1-3.2]{BaladiHachemi}.
\end{proof}

\begin{proof}[Proof of \cref{thm:maintheorem}]
Using \cref{prop:asymptotic:expansion} together with \cref{lem:s0_asymptotic} and \cref{prop:evaloftheintegral}, we have
\begin{align*}
\chi_S(t)&=\exp\bigl(\log X \, (2s_0(t)-2)\bigr)\left(1+\mathcal{O}\!\left(\frac{1}{X^\delta}+|t|\right)\right)\\
&=\exp\Bigl( \log X \bigl( \kappa_D t^2 \log|t| + \mathcal{O}(t^2) \bigr) \Bigr)\left(1+\mathcal{O}\!\left(\frac{1},{X^\delta}+|t|\right)\right)
\end{align*}
for some $\kappa_D>0$. Now, given any $v\in \FFFF(X)$ we have the approximation $\tilde{\mathfrak{D}}(v)=S(v)+\mathcal{O}(1)$ form \eqref{eq:dedekindapprox}. Define the characteristic function
\begin{align*}
    f(t):=\ExpX\!\left[\exp\!\left( it\frac{S(z)}{\sqrt{\kappa_D\log X\log\log X}}\right)\right]
\end{align*}
and set $u:=\frac{t}{\sqrt{\kappa_D \log X\log\log X}}$. Notice that $f(t)=\chi_{X}(u)$. Thus, using \cref{prop:asymptotic:expansion} together with \cref{lem:s0_asymptotic} and \cref{prop:evaloftheintegral}, we have
\begin{align*}
\chi_S(u) = \exp\!\left[ \frac{u^2}{\log\log X} \left( \log \vert u \vert - \frac12 \log\log X + \mathcal{O}(\log\log\log X) \right) \right] \bigl( 1 + g(X,u) \bigr),
\end{align*}
with
\begin{align*}
    g(X,u) = \mathcal{O}\left( \frac{1}{X^\delta} + \frac{|u|}{\sqrt{(\log X)(\log\log X)}}\right).
\end{align*}
Let $0 \le T_1 \le T_2 \le T$ be parameters at our disposal. Selecting $T = (\log\log X)^{1-\epsilon}$ in \cref{lem:Berry-Eessen}, we only need to show
\begin{align*}
    \left(\int_{0}^{T_1} + \int_{T_1}^{T_2} + \int_{T_2}^{T} \right)
    \frac{|e^{-u^2/2} - \chi_S(u)|}{u} \, du \ll \frac{1}{\left(\log\log X\right)^{1-\epsilon}}.
\end{align*}
 For $0 \le u \le T_1 = 1/(\log^{2L} X)$, we estimate using \cref{lem:veryverysmalt},
\begin{align*}
    \int_{0}^{T_1} \frac{|e^{-u^2/2} - \chi_S(u)|}{u} \, du 
    \ll \frac{1}{\log X}.
\end{align*}
For $\frac{1}{\log^{2L} X} \le u \le \log\log\log X$, we have
\begin{align*}
    e^{-u^2/2} - \chi_S(u) 
    &= e^{-u^2/2} \Bigl( 1 - \exp\Bigl[ \frac{u^2}{\log\log X} 
       \bigl( \log u + \mathcal{O}(\log\log\log X) \bigr) \Bigr] \Bigr) \\
    &\ll \frac{e^{-u^2/2} \, u^2 (\log u + \log\log\log X)}{\log\log X},
\end{align*}
and therefore
\begin{align*}
    \int_{T_1}^{T_2} \frac{|e^{-u^2/2} - \chi_S(u)|}{u} \, du 
    &\ll \frac{1}{\sqrt{(\log X)(\log\log X)}} 
       \int_{1/(\log^{2L} X)}^{\log\log\log X} e^{-u^2/2} (u\log u + u) \, du \\
    &\ll \frac{1}{\log\log X}.
\end{align*}
Finally, with $T_2 = \log\log\log X$ we obtain
\begin{align*}
    \int_{\log\log\log X}^{T} \frac{|e^{-u^2/2}-\chi_S(u)|}{u} \, du \ll \frac{1}{\left(\log\log X\right)^{1-\epsilon}},
\end{align*}
which finishes the proof.
\end{proof}

\begin{proof}[Proof of \cref{cor:ourcorollary}]
Since for any $z\in K$ and any $a\in \ringint$ we have $\tilde{\mathfrak{D}}(a+z)=\tilde{\mathfrak{D}}(z)$, it is enough to consider $I_2\cap \{\Im(z)\geq 0\}$. We also know that $\tilde{\mathfrak{D}}(-z)=-\tilde{\mathfrak{D}}(\bar{z})$, see \cite[Theorem 1]{ItoConj}. Then, \cref{thm:maintheorem} ensures the lower bound $\gg \sqrt{\log X\log \log X}$ for the restricted first moment.
\end{proof}

\bibliographystyle{amsplain}
 \bibliography{bibliography}

\end{document}